\documentclass[11pt,reqno]{amsart}
\usepackage{amssymb,amsmath,amsthm,amsfonts}

\makeatletter
\@addtoreset{equation}{section}

\usepackage{amsfonts}
\usepackage{xr}
\usepackage{graphicx}
\usepackage{amsmath}
\usepackage{epsfig}

\def\d{{\partial}}
\def\CC{{\mathbb C}}
\def\C{{\mathbb C}}

\def\ZZ{{\mathbb Z}}

\def\DD{{\mathbb D}}
\def\D{{\mathbb D}}

\def\NN{{\mathbb N}}

\def\RR{{\mathbb R}}
\def\R{{\RR}}

\def\TT{{\mathbb T}}
\def\T0{{\mathbb T}_{x_0}}
\def\T{{\mathbb T}}

\def\trait (#1) (#2) (#3){\vrule width #1pt height #2pt depth #3pt}
\def\fin{\hfill\trait (0.1) (5) (0) \trait (5) (0.1) (0) \kern-5pt 
\trait (5) (5) (-4.9) \trait (0.1) (5) (0)}
\def\tr{\mbox{tr}}

\newtheorem{prop}{Proposition}[section]
\newtheorem{thm}[prop]{Theorem}
\newtheorem{lem}[prop]{Lemma}
\newtheorem{cor}[prop]{Corollary}
\newtheorem{rem}[prop]{Remark}

\begin{document}
\title[\sf Pseudo-holomorphic functions at the critical exponent]{Pseudo-holomorphic functions at the critical exponent}
\author[Baratchart, Borichev, Chaabi]{Laurent Baratchart, Alexander Borichev, Slah Chaabi}
\address{L. Baratchart\\ INRIA Sophia-Antipolis, 
BP 93, 06902 Sophia-Antipolis Cedex, France}
\email{Laurent.Baratchart@inria.fr}
\address{A. Borichev\\ LATP, Aix-Marseille Universit\'e\\
39, rue F. Joliot-Curie, 13453 Marseille, France}
\email{borichev@cmi.univ-mrs.fr}
\address{S. Chaabi \\INRIA Sophia-Antipolis, 
BP 93, 06902 Sophia-Antipolis Cedex, France}
\email{Slah.Chaabi@inria.fr}
\date{\today}
\subjclass[2000]{
}
\thanks{}
\begin{abstract} {We study Hardy classes on the 
disk associated to the equation $\bar\d w=\alpha\bar w$ for $\alpha\in L^r$ with 
$2\leq r<\infty$. The paper seems to be the first to deal with the case $r=2$.
We prove an analog of the M.~Riesz theorem and a 
topological converse to the Bers similarity principle. 
Using the connection between
pseudo-holomorphic functions and conjugate Beltrami equations, we
deduce well-posedness on smooth domains of the
Dirichlet problem with weighted $L^p$ boundary data for 2-D isotropic
conductivity equations whose coefficients have logarithm in $W^{1,2}$.
In particular these are not strictly elliptic.
Our  results depend on a
new multiplier theorem for $W^{1,2}_0$-functions.}
\end{abstract}
\maketitle

\section{Introduction} 

\label{sec:pb}
Pseudo-holomorphic functions of one complex variable, 
{\it i.e.} solutions to a $\bar \d$ equation 
whose right-hand side is a real linear function of the unknown variable,  are
perhaps the simplest generalization  of holomorphic functions.
They received early attention in \cite{Theodorescu,Carleman}
and extensive treatment in \cite{Bers,vekua}
when the coefficients are $L^r$-summable,
$r>2$  While \cite{Bers}
takes on a  function-theoretic viewpoint, \cite{vekua} dwells on
integral equations and leans on applications to geometry, 
elasticity and hydrodynamics. 
Recent developments and applications 
to various boundary value problems can be found in \cite{krav,Wen,EfWend}.
Hardy classes for such functions were introduced in
\cite{Musaev1} and subsequently considered in
\cite{Klimentov1,Klimentov2,klimentov2006,BLRR} in the range of exponents $1<p<\infty$, see 
\cite{EfendievRuss,klimentov2011, TheseYannick,BFL} for 
further generalizations
to multiply connected domains. The connection between pseudo-holomorphic functions and conjugate Beltrami equations makes 
pseudo-holomorphic Hardy classes a convenient framework to solve
Dirichlet problems with $L^p$ boundary data for isotropic
conductivity equations \cite{BLRR,EfendievRuss,BFL}.  These are also
instrumental in \cite{fl,flps,FMP, TheseYannick} to approach certain inverse 
boundary  problems. 

As reported in \cite{Bliev}, I. N. Vekua stressed on several 
occasions an interest in developing the theory for $L^r$ coefficients 
when  $1<r\le 2$. However, 
solutions then need no longer be continuous which has apparently been an 
obstacle
to such extensions, see \cite{Bliev,Otelbaev} for classes of coefficients 
that ensure such continuity.
The present paper seems to be the first to deal with the critical exponent 
$r=2$. 
We develop a theory of pseudo-holomorphic Hardy spaces on the disk
in the range $1<p<\infty$, prove  
existence of $L^p$ boundary values,  and give 
an analog of the M. Riesz theorem in this context. 
As a byproduct, we obtain a Liouville-type theorem.
We also develop a 
topological parametrization by holomorphic Hardy functions which is new even 
for  $r>2$. We apply our result to 
well-posedness of the Dirichlet problem
with weighted $L^p$ boundary data for 2-D conductivity equations whose
coefficients have logarithm in $W^{1,2}$. In particular these are not bounded 
away from zero nor infinity and no strict ellipticity prevails, 
which makes for results of a novel type. 
Accordingly, solutions may be
locally unbounded.  

As in previous work on pseudo-holomorphic functions, we make extensive use of 
the Bers similarity principle, but in our case it requires a thorough analysis of
smoothness and boundedness properties of exponentials of $W^{1,2}$ 
functions which is carried 
out in a separate appendix. There we prove a theorem,
one of the main technical results of the paper, asserting that the exponential of a $W^{1,2}_0$ 
function in the disk is a multiplier from the space of functions with $L^p$ 
maximal function on the unit circle to the space of functions satisfying a Hardy 
condition of order $p$ on the unit disk. This would have higher dimensional analogs, but we 
make no attempt at developing them and stick to dimension 2 throughout the 
paper.

In Section~\ref{sec:notations_generales} we introduce main notations and discuss numerous facts on Sobolev spaces we use later on. 
In Section~\ref{x18} we formulate the classical similarity principle (factorization) for pseudo-holomorphic functions.
A converse statement is given in Section~\ref{x19}. Section~\ref{HSD} is devoted to pseudo-holomorphic Hardy spaces; we give there 
a topological converse to the similarity principle. In Section~\ref{x20} we obtain a generalization of the M. Riesz theorem on the conjugate operator. 
Section~\ref{sectDirichlet} contains an application of our results to the conductivity equation with exp-Sobolev coefficients. Finally, several technical results 
and a multiplier theorem are contained in the appendix, Section~\ref{app}.

\section{Notations and preliminaries}
\label{sec:notations_generales}

Let $\C\sim\R^2$ be the complex plane and 
$\overline{\C}:=\C\cup\{\infty\}$.
We designate by $\T_{\xi,\rho}$ and $\D_{\xi,\rho}$ respectively the circle 
and the open  disk centered at $\xi$ of radius $\rho$. We simply write
$\T_\rho$, $\D_\rho$ when $\xi=0$, and 
if $\rho=1$ we omit the subscript. 
If $f$ is a function on 
$\D_\rho$, we often denote by 
$f_\rho$ the function on $\D$ defined by $f_\rho(\xi):=f(\rho\xi)$. 
Given $\xi\in\T$ and $\gamma\in(0,\pi/2)$, 
we let $\tilde\Gamma_{\xi,\gamma}$ indicate the open cone 
with vertex $\xi$ and opening $2\gamma$, symmetric with respect 
to the line $(0,\xi)$. 
We define $\Gamma_{\xi,\gamma}=A_{\xi,\gamma}\cup\bar\D_{\sin\gamma}$, where $A_{\xi,\gamma}$ is the bounded component of 
$\tilde\Gamma_{\xi,\gamma}\setminus\bar\D_{\sin\gamma}$.

A complex-valued function $f$ on $\D$ has  non-tangential limit
$\ell$  at $\xi$ if
$f(z)$ tends to $\ell$ as
$z\to \xi$ inside $\Gamma_{\xi,\gamma}$ for every $\gamma$.
The non-tangential maximal function
of $f$ (with opening $2\gamma$) is the real-valued map  ${\mathcal M}_\gamma f$
on
$\T$ given by
\begin{equation}
\label{defntmax}
{\mathcal M}_{\gamma}f(\xi):=\sup_{z\in \D\cap\Gamma_{\xi,\gamma}}
|f(z)|,\qquad \xi\in\T.
\end{equation}
For
$E\subset\C$ and $f$ a function on a set containing  $E$, we let 
$f_{|E}$ indicate the restriction of $f$ to $E$. 
We put $|E|$ for the planar Lebesgue measure of $E$, 
as no confusion can arise with 
complex modulus.
The differential of that measure is denoted 
interchangeably by
$$
dm(z) = dx\,dy = (i/2) \ dz \wedge d \overline{z},\qquad z=x+iy.
$$

\medskip
When $\Omega\subset\C$ is an open set, we denote by ${\mathcal D}(\Omega)$ the space of 
$C^\infty$-smooth
complex-valued 
functions with compact support in $\Omega$, equipped with the usual
topology\footnote{{\it i.e.} the inductive topology of subspaces  
$\mathcal{D}_K$ consisting of functions supported by the compact set 
$K$, each $\mathcal{D}_K$ being topologized by uniform convergence of 
all derivatives \cite[Section I.2]{Schwartz}.}.
Its dual $\mathcal{D}'(\Omega)$ 
is the space of distributions on $\Omega$. 
For $p\in[1,\infty]$,  
we let $L^p(\Omega)$ and $W^{1,p}(\Omega)$ 
be the usual Lebesgue and Sobolev spaces  with respect to $dm$;
we sometimes consider their subspaces of real-valued functions $L^p_\RR(\Omega)$ and $W_\RR^{1,p}(\Omega)$.
The space $W^{1,p}(\Omega)$ consists of functions in
$L^p(\Omega)$ whose first distributional derivatives lie in
$L^p(\Omega)$, with the norm:
$$
 \|f\|_{W^{1,p}(\Omega)} = \|f\|_{L^p(\Omega)}+\|\partial f\|_{L^p(\Omega)}+\|\overline{\partial} f\|_{L^p(\Omega)}.
$$
Here $\partial$ and $\overline{\partial}$ stand for the usual
complex derivatives:
$$
 \partial f := \partial_z f = \dfrac{1}{2}(\partial_x - i \partial_y)f~~~\mbox{and}~~~\overline{\partial} f := \partial_{\overline{z}} f = \dfrac{1}{2}(\partial_x + i \partial_y)f, \qquad z=x+iy.
$$
Setting $\nabla f:=(\partial_x f,\partial_y f)$  to mean the
($\C^2$-valued) gradient of $f$, observe that the pointwise relation 
$\|\nabla f \|_{\C^2}^2=2|\partial f|^2_2+2|\bar \partial f|_2^2$ 
holds. 
Note also the identities $\overline{\partial f}=\overline{\partial} 
\,\,\overline{f}$ and $\Delta=4\d\bar\d$, where $\Delta$ is the 
Euclidean Laplacian.  By Weyl's lemma \cite[Theorem 24.9]{Forster},
the distributions $u\in \mathcal{D}'(\Omega)$ such 
that $\Delta u=0$ are exactly the harmonic functions on $\Omega$.
Subsequently, the distributions $\psi\in \mathcal{D}'(\Omega)$ such 
that $\bar\d\psi=0$ are exactly the holomorphic functions on $\Omega$.
The space $\mathcal{D}(\R^2)$ is dense in $W^{1,p}(\R^2)$ for
$p\in[1,\infty)$, and in general 
we let $W^{1,p}_0(\Omega)$ indicate
the closure of $\mathcal{D}(\Omega)$ in $W^{1,p}(\Omega)$.
The space $W^{1,\infty}(\Omega)$ identifies with 
Lipschitz-continuous functions on $\Omega$
\cite[Section V.6.2]{stein}.

We also introduce the spaces $L^p_{loc}(\Omega)$ and $W^{1,p}_{loc}(\Omega)$ 
of distributions whose restriction to any relatively compact open subset 
$\Omega_0\subset \Omega$ lies in $L^p(\Omega_0)$ and 
$W^{1,p}(\Omega_0)$ respectively. 
They are topologized by the family of 
semi-norms $\|f_{\Omega_n}\|_{L^p(\Omega_n)}$ and
$\|f_{\Omega_n}\|_{W^{1,p}(\Omega_n)}$,  where $\{\Omega_n\}$ is
a sequence of relatively compact open subsets exhausting $\Omega$. 

Below we indicate some properties of Sobolev functions, most of them  
standard. They are
valid on bounded  Lipschitz domains 
({\it i.e.} domains $\Omega$  whose boundary 
$\partial\Omega$ is locally isometric to the graph of a Lipschitz function).

\begin{itemize}
\item
For $1\leq p\leq\infty$, every $f\in W^{1,p}(\Omega)$ is the restriction 
to $\Omega$ of some $\tilde{f}\in W^{1,p}(\R^2)$. In fact, there is a 
continuous  linear map 
\begin{equation}
\label{x22}
E: W^{1,p}(\Omega)\to W^{1,p}(\R^2) \text{\ such that\ } (Ef)_{|\Omega}=f
\end{equation}
(the extension theorem  \cite[Proposition 2.70]{Demengel}).
When $\Omega=\D_\rho$,
we may simply put $(Ef)_{|\C\setminus\D_\rho}(z)=\varphi(z) f(\rho^2/\bar z)$,  
where
$\varphi\in\mathcal{D}(\R^2)$ and $\varphi_{|\D_\rho}\equiv1$.
The extension theorem entails that
smooth functions on $\overline{\Omega}$ are dense in $W^{1,p}(\Omega)$ when
$1\leq p<\infty$. 

\item For $p>2$, $W^{1,p}(\Omega)$ embeds continuously
in the space of  H\"older-smooth functions  with 
exponent  $1-2/p$ on $\Omega$, 
in particular functions in $W^{1,p}(\Omega)$ 
extend continuously to
$\overline{\Omega}$, and $W^{1,p}(\Omega)$ is an algebra where multiplication is continuous and
derivatives 
can be computed by the chain rule. For $1\le p<2$ the embedding is in
$L^{p^*}(\Omega)$ with $p^*=2p/(2-p)$, while $W^{1,2}(\Omega)$ is
embedded in all $L^\ell(\Omega)$, $\ell\in[1,\infty)$ 
(the Sobolev embedding theorem \cite[Theorems 4.12,4.39]{Adams}). 

\item
For $p\leq2$ the embedding $W^{1,p}(\Omega)\to L^\ell(\Omega)$ is compact when 
$\ell\in[1,p^*)$ (the Rellich--Kondrachov theorem \cite[Theorem 6.3]{Adams}); $p^*=\infty$ for $p=2$.

\item  If $g\in \mathcal{D}'(\Omega)$
has derivatives in $L^p(\Omega)$  for some $p\in[1,\infty)$, then
$g\in W^{1,p}(\Omega)$ \cite[Theorem 6.74]{Demengel}\footnote{The proof given 
there for bounded $C^1$-smooth $\Omega$ carries over to 
the Lipschitz case.}. 
Moreover, there exists    
$C=C(\Omega,p)$ such that 
\begin{equation}
\label{estfonc}
\left\|g-g_\Omega\right\|_{L^p(\Omega)}\leq C\bigl(\|\partial g\|_{L^p(\Omega)}+
\|\bar \partial g\|_{L^p(\Omega)}\bigr),\,\,\,\mbox{with}\,\,\, g_\Omega:=\frac{1}{|\Omega|}
\int_\Omega g\,dm
\end{equation}
(the Poincar\'e inequality  \cite[Theorem 4.2.1]{ziemer}).
Let $C_1=C_1(p)$ be a number for which
\eqref{estfonc} holds for $\Omega=\D$;  
it is easily seen by homogeneity that if $\xi\in\C$, $\rho>0$, and
$g\in W^{1,p}(\D_{\xi,\rho})$, then
\begin{multline}
\label{homQprim}
\Bigl(\frac{1}{|\D_{\xi,\rho}|}\int_{\D_{\xi,\rho}}|g-g_{\D_{\xi,\rho}}|^p\,dm
\Bigr)^{1/p}
\\ \le C_1\rho^{1-2/p}\bigl(\|\partial g\|_{L^p(\D_{\xi,\rho})}+
|\bar\partial g\|_{L^p(\D_{\xi,\rho})}).
\end{multline}
In particular, if $p=2$ and $\partial g, \bar\partial g\in L^2(\Omega)$, 
then the right hand side of \eqref{homQprim} is bounded and
arbitrarily small as $\rho\to0$, thereby asserting that
$g$ lies in $VMO(\Omega)$, the space of functions with vanishing mean 
oscillation on $\Omega$ \cite{BrN1}.

\item $W^{1,p}(\Omega)$-functions need not be continuous nor even 
locally bounded when $p\le 2$; however, if $p>1$, their
non-Lebesgue points form a 
set of Bessel $B_{1,p}$-capacity zero \cite[Theorem 3.10.2]{ziemer}.
Such sets are very thin: not only do they have measure zero but also 
their Hausdorff 
$H^{2-p+\varepsilon}$-dimension is zero for each $\varepsilon>0$
\cite[Theorem 2.6.16]{ziemer}.  
When speaking of pointwise values of $f\in W^{1,p}(\Omega)$,
we pick a representative such that
$f(z)=\lim_{\varepsilon\to0}f_{\D_{z,\varepsilon}}$ outside a set of
 $B_{1,p}$-capacity zero. At such a $z$, $f$ is said to be 
\emph{strictly defined}.

\item If $L^\lambda(\partial\Omega)$ is understood with respect 
to arclength, then $W^{1,\lambda}(\partial\Omega)$ 
is naturally defined using 
local coordinates  
since any Lipschitz-conti\-nuous change of variable preserves Sobolev classes
\cite[Theorem 2.2.2]{ziemer}.
Each 
$f \in W^{1,p}(\Omega)$ with $1<p\leq\infty$  has a trace on 
$\partial \Omega$ (denoted again by $f$ or sometimes
by $\tr_{\partial \Omega} \, f$ for emphasis), which lies  in the Sobolev 
space
$W^{1-1/p,p}(\partial\Omega)$ of non-integral order\footnote{We leave out the 
case $p=1$ where the trace is merely defined in $L^1(\partial\Omega)$.
The space $W^{1-1/p,p}(\partial\Omega)$  coincides with
the Besov space $B^{1-1/p,p}_p(\partial\Omega)$, but we need not introduce 
Besov spaces here.}.  
The latter is a real interpolation space
between $L^p(\partial\Omega)$ and $W^{1,p}(\partial\Omega)$,
with the norm given by \cite[Theorem 7.47]{Adams}:
\begin{equation}
\label{defWt}
\|g\|_{W^{1-1/p,p}(\partial\Omega)}=\|g\|_{L^p(\partial\Omega)}+
\Bigl(\int_{\partial\Omega\times\partial\Omega}
\!\frac{|g(t)-g(t')|^p}{\bigl(\Lambda(t,t')\bigr)^p}\,d\Lambda(t)d\Lambda(t')
\Bigr)^{1/p},
\end{equation}
where $\Lambda(t,t')$ indicates the length of the arc $(t,t')$ on $\partial\Omega$.
Note that $|t-t'|\sim\Lambda(t,t')$ since
$\partial\Omega$ is Lipschitz. The trace operator defines a continuous
surjection from
$W^{1,p}(\Omega)$ onto  $W^{1-1/p,p}(\partial\Omega)$ 
\cite[Theorem 1.5.1.3]{gris}. 
The pointwise definition of $\tr_{\partial\Omega} f$  $\Lambda$-{\it a.e.}
is based on the extension theorem 
and the fact that non-Lebesgue points of $E f$ (see \eqref{x22}) have  Hausdorff $H^1$-measure 
zero \cite[Remark 4.4.5]{ziemer}. Of course 
$\tr_{\partial \Omega} \ f$ coincides with the restriction 
$f_{|\partial\Omega}$ whenever $f$ is smooth on $\overline{\Omega}$.
The subspace of functions with zero trace is none but
$W^{1,p}_0(\Omega)$.

Since the integral in the right hand side of \eqref{defWt} does not change if
we add a constant to $g$, it follows from \eqref{estfonc} by
the continuity of  the trace operator that
\begin{equation}
\label{majvarder}
\Bigl(\int_{\partial\Omega\times\partial\Omega}
\frac{|g(t)-g(t')|^p}{\bigl(\Lambda(t,t')\bigr)^p}\,d\Lambda(t)d\Lambda(t')
\Bigr)^{1/p}\le  C\Bigl
(\|\partial g\|_{L^p(\Omega)}+
\|\bar \partial g\|_{L^p(\Omega)}\Bigr),
\end{equation}
where the constant $C$ depends on $\Omega$ and $p$.

A variant of the Poincar\'e inequality involving the trace is as follows: whenever
$E\subset\partial\Omega$ has arclength $\Lambda(E)>0$,
there is 
$C>0$ depending only on $p$, $\Omega$ and $E$ such that
\begin{equation}
\label{estfoncb}
\Bigl\|g-\int_E\tr_{\partial\Omega} \,g\Bigr\|_{L^p(\Omega)}\le C\bigl
(\|\partial g\|_{L^p(\Omega)}+\|\bar \partial g\|_{L^p(\Omega)}\bigr).
\end{equation}
This follows immediately from the continuity of the trace operator, the 
Rellich--Kondrachov theorem,  and \cite[Lemma 4.1.3]{ziemer}.

\item For $p\in(1,\infty)$ 
the trace operator has a continuous section \cite[Theorem 1.5.1.3]{gris},
that is, for each 
$\psi\in W^{1-1/p,p}(\partial\Omega)$,
there is $g\in W^{1,p}(\Omega)$ such that
\begin{equation}
\label{bornsec}
\|g\|_{W^{1,p}(\Omega)}\leq C\|\psi\|_{W^{1-1/p,p}(\partial\Omega)},
\qquad \tr_{\partial\Omega} \,g=\psi,
\end{equation}
with $C=C(\Omega, p)$. 
If we assume that $\Omega$ is $C^1$-smooth
and not just Lipschitz,
then the function $g$ in \eqref{bornsec} can be chosen to be 
harmonic in $\Omega$
(elliptic regularity theory 
\cite[p.165 \& Theorem 1.3]{JK})\footnote{In fact, elliptic regularity holds
for $1<p<\infty$ 
as soon as $\partial\Omega$ is locally the graph of a function
with VMO derivative \cite[Theorem 1.1]{MMS}.  If $\partial \Omega$ is only Lipschitz-smooth, then 
the range of $p$ has to be 
restricted in a manner that depends on the Lipschitz constant,
see \cite{JK,MMS}.}.

\item The non-integral version of the Sobolev embedding theorem
\cite[Theorem 7.34]{Adams} asserts that $W^{1-1/\beta,\beta}(\partial\Omega)$
embeds continuously in \newline\noindent$L^{ \beta/(2 - \beta)}(\partial\Omega)$ if
$1<\beta<2$, while $W^{1/2,2}(\partial\Omega)$ embeds in 
$L^\ell(\partial\Omega)$ for  all $\ell \in [1, \infty)$. 
The corresponding generalization
of the
Rellich--Kondrachov theorem \cite[Theorem 4.54]{Demengel} is as follows:
if $1<\beta\leq 2$, then
$W^{1-1/\beta,\beta}(\partial\Omega)$ embeds compactly in 
$L^\ell(\partial\Omega)$ for $\ell<{\beta/(2 - \beta)}$.
 
\item When $p\in(2,\infty)$, the nonlinear map
$f\mapsto e^f$ is bounded and continuous from $W^{1,p}(\Omega)$ into itself: 
this follows from the Taylor expansion of $\exp$ because 
$W^{1,p}(\Omega)$ is an algebra.
When $p=2$ this property no longer holds, but still $f\mapsto e^f$ 
is continuous and bounded from $W^{1,2}(\Omega)$ into $W^{1,q}(\Omega)$ for 
each $q\in[1,2)$; in particular 
$\tr_{\partial\Omega} e^f=e^{\tr_{\partial\Omega} f}$ 
exists in $W^{1-1/q,q}(\partial\Omega)$ for
$1<q<2$.
This is the content of Proposition~\ref{embexpS} that we could 
not locate in the literature.
\item
We use at some point the Sobolev space $W^{2,p}(\Omega)$ of functions in
$L^p(\Omega)$ whose first distributional derivatives lie in
$W^{1,p}(\Omega)$, equipped with the norm:
$$
\|f\|_{W^{2,p}(\Omega)} = \|f\|_{L^p(\Omega)}+
\|\partial f\|_{W^{1,p}(\Omega)}+\|\overline{\partial} f\|_{W^{1,p}(\Omega)}.
$$
When $p\leq2$, the Rellich--Kondrachov theorem implies 
that  $W^{2,p}(\Omega)$ is compactly embedded in $W^{1,\ell}(\Omega)$
for $\ell\in[1,p^*)$.
\end{itemize}

Given a bounded domain $\Omega$ and $h\in L^p(\Omega)$, $1<p<\infty$, 
let  $\tilde{h}$ denote  the extension of $h$ by $0$ off $\Omega$.
The Cauchy integral operator applied 
to $\tilde{h}$
defines a function  $\mathcal{C}(h)\in W^{1,p}_{loc}(\R^2)$ given by
\begin{equation}
\label{defCauchy}
\mathcal{C}(h)(z)=\frac{1}{\pi}\int_{\Omega}\frac{h(t)}{z-t}dm(t)=
\frac{1}{2\pi i}\int_{\Omega}\frac{h(\xi)}{\xi-z}d\xi\wedge d\xi,
\qquad z\in\C.
\end{equation}
Indeed, $\mathcal{C}(h)$ lies in $L^1_{loc}(\C)$ by Fubini's theorem.
Furthermore, $z\mapsto1/(\pi z)$ is a 
fundamental solution of the $\bar\partial$ operator and it follows
that $\bar\partial\mathcal{C}(h)=\tilde{h}$ in the sense of distributions.
In another connection (see \cite[Theorem 4.3.10]{aim} and the remark thereafter),
the complex derivative $\partial\mathcal{C}(h)$ is given by the singular 
integral
\begin{equation}
\label{defB}
\mathcal{B}(h)(z):=\lim_{\varepsilon\to0} \,-\frac{1}{\pi}
\int_{\Omega\setminus D(z,\varepsilon)} \frac{h(\xi)}
{(z-\xi)^2}dm(\xi),\qquad z\in\C,
\end{equation}
which is the so-called Beurling transform of $\tilde{h}$.
By a result of Calder\`on and Zygmund (see \cite[Theorem 4.5.3]{aim}) this 
transform maps $L^p(\C)$ continuously into itself, and altogether we conclude
that $\mathcal{C}(h)\in W^{1,p}_{loc}(\C)$, as announced.

The discussion above shows in particular that 
$\varphi:=\mathcal{C}(h)_{|\Omega}$ lies in $W^{1,p}(\Omega)$, and that
\[\|\partial\varphi\|_{L^p(\Omega)}+\|\bar\partial\varphi\|_{L^p(\Omega)}=
\|\mathcal{B}(h)_{|\Omega}\|_{L^p(\Omega)}+
\|h\|_{L^p(\Omega)}\leq c\|h\|_{L^p(\Omega)},\]
where $c$ depends only on $p$. In addition, it is a consequence
of Fubini's theorem that $\|\varphi\|_{L^p(\Omega)}
\leq 6\,\mbox{\rm diam}\,\Omega\,\|h\|_{L^p(\Omega)}$  
\cite[Theorem 4.3.12]{aim}. Therefore, we have 
\begin{equation}
\label{estNSC}
\left\|\mathcal{C}(h)\right\|_{W^{1,p}(\Omega)}\leq C\|h\|_{L^p(\Omega)},
\end{equation}
where $C$ depends only on $p$ and $\Omega$. Moreover, 
if 
$\Omega\subset \mathbb D_R$, then $\mathcal{C}(h)$ coincides 
on $\Omega$ with the convolution of $\tilde{h}$ with
$z\mapsto \chi_{\mathbb D_{2R}}(z)/z$, 
where $\chi_E$ denotes the characteristic function of a set $E$. 
Therefore 
$\d \mathcal{C}(\varphi)_{|\Omega}=\mathcal{C}(\d\varphi)_{|\Omega}$
whenever $\varphi\in\mathcal{D}(\Omega)$, and by density argument it follows that 
\begin{equation}
\label{CauchyW12}
\|\mathcal{C}(h)\|_{W^{2,p}(\Omega)}\leq C\|h\|_{W^{1,p}(\Omega)},
\qquad h\in W^{1,p}_0(\Omega),
\end{equation}
for $p\in(1,\infty)$ and some $C=C(p,\Omega)$.

Properties of the Cauchy transform  make 
it a basic tool to integrate $\bar\partial$-equations in Sobolev 
classes. In this connection, we record the following facts.
\begin{itemize}
\item 
Given a bounded open set $\Omega\subset\C$  and $a\in L^p(\Omega)$
with $p\in(1,\infty)$,  a distribution $A\in \mathcal{D}'(\Omega)$ 
satisfies $\bar\partial A=a$ if and  only if 
$A=\mathcal{C}(a)+\Phi$ where $\Phi$ is holomorphic in $\Omega$.
This follows from  the relation $\bar\partial\mathcal{C}(a)=a$ and
Weyl's lemma. By \eqref{estNSC},
$A$ belongs to  $W^{1,p}(\Omega)$ if and only if
$\Phi$ does. By localization, it follows that if 
$f\in L^1_{loc}(\Omega)$ satisfies $\bar\partial f\in L^p_{loc}(\Omega)$,
then $f\in W^{1,p}_{loc}(\Omega)$.
\item  Given a bounded $C^1$-smooth simply connected  domain 
$\Omega\subset\C$   and  $a\in L^p(\Omega)$
with $p\in(1,\infty)$,
for every $\psi\in W^{1-1/p,p}(\partial\Omega)$,
$\lambda\in\R$, $\theta_0\in\R$, there exists a unique $A\in W^{1,p}(\Omega)$
such that $\bar\partial A=a$ with 
$\tr_{\partial\Omega}\mbox{\,\rm Re}\,(e^{i\theta_0}A)=\psi$,
and $\int_{\partial\Omega} \mbox{\rm Im }(e^{i\theta_0}A)=\lambda$.
Moreover, there exists $C$ depending only 
on $p$ and $\Omega$ such that 
\begin{equation}
\label{intdbar}
\|A\|_{W^{1,p}(\Omega)}\le C \bigl(\|a\|_{L^p(\Omega)}+
\|\psi\|_{W^{1-1/p,p}(\partial\Omega)}+|\lambda|\bigr).
\end{equation}
To see this, it suffices,
in view of \eqref{estNSC} and the previous remark,
to consider the case  $a=0$. Clearly, we may also  
assume that $\theta_0=0$. 
By elliptic regularity, there is a unique
$u\in W_\R^{1,p}(\Omega)$,  harmonic in $\Omega$ and such that
$\tr_{\partial\Omega}u=\psi$. Moreover, $u$ satisfies
$\|u\|_{W^{1,p}(\Omega)}\leq C\|\psi\|_{W^{1-1/p,p}(\partial\Omega)}$.
As $\Omega$ is simply connected, integrating the conjugate 
differential yields a  so-called
harmonic conjugate to $u$, that is a
real-valued harmonic function $v$, 
such that $A:=u+iv$ is holomorphic in $\Omega$.
Since $u$ and $v$ are real, the Cauchy--Riemann equations give  
$|\partial v|=|\bar\partial v|=|\partial u|$. Hence, we have $v\in W^{1,p}_\R(\Omega)$. 
Clearly $v$ is unique up to an additive 
constant, and if $\int_{\partial\Omega}v=\lambda$ we deduce from
\eqref{estfoncb}
that
$\|v\|_{W^{1,p}(\Omega)}\leq C_1\|u\|_{W^{1,p}(\Omega)}+c_1|\lambda|$
so that \eqref{intdbar} holds (with $a=0$),
as desired.
\end{itemize}
When $h\in L^2(\C)$ has unbounded support, definition
\eqref{defCauchy} of the Cauchy transform is no longer suitable.
Instead, one renormalizes the kernel and defines 
\begin{equation}
\label{defCauchy2}
\mathcal{C}_2(h)(z):=\frac1\pi\int_{\R^2}\Bigl(\frac1{z-t}+
\frac{\chi_{\C\setminus\D}(t)}{t}\Bigr) h(t)\,dm(t),
\qquad z\in\C.
\end{equation}
Since $h\in L^2(\C)$, the integral in \eqref{defCauchy2} 
converges 
for a.e. $z\in\C$ by Fubini's theorem and the Schwarz inequality. In fact, 
the function $\mathcal{C}_2(h)$ 
belongs to the space $VMO(\C)$ \cite[Theorem 4.3.9]{aim}. 
Furthermore, 
$\bar\partial\mathcal{C}_2(h)=h$ and 
$\partial\mathcal{C}_2(h)=\mathcal{B}(h)$
\cite[Theorem 4.3.10]{aim}. In particular, $\mathcal{C}_2(h)$ 
lies in $W^{1,2}_{loc}(\C)$ and
the map $h\mapsto\mathcal{C}_2(h)$ maps $L^2(\C)$ continuously
into $W^{1,2}_{loc}(\C)$.

In Section  \ref{croissC2}  we prove the following estimate, 
valid for some absolute constant $C$:
\begin{equation}
\label{estdRC2}
\frac{\|\mathcal{C}_2(h)\|_{L^2(\D_R)}}{R}\le 
C\bigl (1+(\log R)^{1/2}\bigr)
\|h\|_{L^2(\D_R)}, \qquad R\geq1.
\end{equation}

Hereafter, all classes of functions 
we consider are embedded in $L_{loc}^p(\Omega)$ 
for some $p \in (1,+\infty)$, and solutions to differential
equations are understood in the distributional 
sense. 

On the disk, we often use the elementary fact that if $f\in W^{1,p}(\D)$, 
then $f_\rho$ converges to $f$ in $W^{1,p}(\D)$ as $\rho\to1^-$.
\medskip
 
Here and later on we use the same symbols (like $C$) to denote different constants.

\section{Pseudo-holomorphic functions}
\label{x18}

Pseudo-holomorphic functions 
on an open set $\Omega\subset\C$ are those functions $\Phi$ that 
satisfy an equation of the form
\begin{equation}
\label{defps}
{\bar\partial }\Phi(z)=a(z)\overline{\Phi(z)}+b(z)\Phi(z),
\qquad z\in\Omega.
\end{equation}
We restrict ourselves to the case where $\Omega$ is bounded and
$a,b\in L^r(\Omega)$ for some $r\in[2,\infty)$. Accordingly, we only consider 
solutions $\Phi$ which belong to $L^\gamma_{loc}(\Omega)$ for some
$\gamma>r/(r-1)$, so that, by H\"older's inequality, the right hand side of \eqref{defps} defines a 
function in $L^{\lambda}_{loc}(\Omega)$ for some $\lambda>1$. As a consequence, 
$\Phi$ belongs to $W^{1,\lambda}_{loc}(\Omega)$.

Let $B\in W^{1,r}(\Omega)$ be such that $\bar\partial B=b$. A simple computation (using Proposition~\ref{embexpS}  
if $r=2$) shows that $\Phi$ satisfies \eqref{defps} if and only if $w:=e^{-B}\Phi$ satisfies
\begin{equation}
\label{eq:w}
\bar\partial w=\alpha\overline{w}, 
\end{equation}
where $\alpha:=ae^{-2i{\rm Im}B}$ has the same modulus as $a$.
Note (again from Proposition~\ref{embexpS} for $r=2$) 
that $w\in W^{1,\lambda'}_{loc}(\Omega)$
for some $\lambda'>1$. Therefore,  by the Sobolev embedding theorem,
$w$ lies in $L^{\gamma'}_{loc}(\Omega)$ for some $\gamma'>2$,
and  so  equation \eqref{eq:w} is a simpler
but equivalent form of \eqref{defps} which is the one we shall really 
work with.

We need a factorization principle
which goes back to \cite{Theodorescu}, and was called by Bers 
the similarity principle (similarity to holomorphic functions, that is). 
It was extensively used
in all works mentioned above.
We provide a proof because we include the case $r=2$ and discuss 
normalization issues when $\Omega$ is smooth.

\begin{lem}[Bers Similarity principle]
\label{expsf2}
Let $\Omega \subset\C $ be a bounded 
domain, $\alpha\in L^r(\Omega)$  for some $r\in[2,\infty)$, 
and $w\in L^{\gamma}_{loc}(\Omega)$ be a solution to
\eqref{eq:w} with  $\gamma>r/(r-1)$. Then 
\begin{itemize}
\item[(i)]
The function $w$ admits a factorization of the form
\begin{equation}
\label{decompexpH2}
w=e^sF,\qquad z \in \Omega,
\end{equation}
where $F$ is holomorphic in $\Omega$, $s \in W^{1,r}(\Omega)$ with
\begin{equation}
\label{regs2}
\|s\|_{W^{1,r}(\Omega)} 
\leq C\|\alpha\|_{L^r(\Omega)},
\end{equation}
and $C$ depends only on $r$ and $\Omega$.
\item[(ii)]
Assume in addition that $\Omega$ is $C^1$-smooth. If $w\not\equiv0$ and
we fix some $\psi\in W^{1-1/r,r}_\R(\partial\Omega)$,
$\lambda\in\R$, and $\theta_0\in\R$, then $s$ can be uniquely chosen
in \eqref{decompexpH2} 
so that $\tr_{\partial\Omega}\mbox{\rm Re}(e^{i\theta_0}s)=\psi$
and $\int_{\partial\Omega} \mbox{\rm Im }(e^{i\theta_0}s)=\lambda$.
In this case, there is a constant $C$ depending only 
on $r$ and $\Omega$ such that 
\begin{equation}
\label{intdbarb}
\|s\|_{W^{1,r}(\Omega)}\le C \bigl(\|\alpha\|_{L^r(\Omega)}+
\|\psi\|_{W^{1-1/r,r}(\partial\Omega)}+|\lambda|\bigr).
\end{equation} 
\item[(iii)]
Either
$w\equiv0$ or $w\neq0$ a. e. on $\Omega$\footnote{In fact, more is true: if 
$r>2$, then $e^s$ never vanishes and $w$ has at most countably many zeros, 
namely those of $F$. If $r=2$, $w$ is strictly defined and nonzero
outside a set of Bessel $B_{1,2}$-capacity
zero (containing the zeros of $F$ and the non Lebesgue
points of $s$).}. Moreover, 
$w\in W^{1,r}_{loc}(\Omega)$ if $r>2$ and 
$w\in W^{1,q}_{loc}(\Omega)$ for all $q\in[1,2)$ if $r=2$. 
\end{itemize}
\end{lem}

\begin{proof}
We pointed out already that
$w\in W^{1,\ell}_{loc}(\Omega)$ for some $\ell>1$.
Set by convention $\overline{w(\xi)}/w(\xi)=0$ if $w(\xi)=0$, and let
$s:=\mathcal{C}(\alpha \bar w/w)_{|\Omega}$.
Then $s\in W^{1,r}(\Omega)$ with
$\bar\partial s=\alpha\overline{w}/w$, and 
\eqref{estNSC} yields \eqref{regs2}. 
To show that $F=e^{-s}w$ is in fact 
holomorphic, we compute
\[\bar\partial(e^{-s}w)=e^{-s}\left(-\bar\partial s \,w+\bar\partial w\right)=
e^{-s}\left(-\frac{\alpha\bar w}{w}w+\alpha\bar w\right)=0,\]
where the use of the Leibniz and the chain rules  
is justified by Proposition~\ref{embexpS} if $r=2$.  This proves $(i)$.

Since $s$ is finite a.e.\ on $\Omega$ (actually outside of a set of
$B_{1,2}$-capacity zero),  $e^s$ is a.e. nonzero and so is $w$
unless the holomorphic function $F$ is identically zero. 
If $r>2$, then $e^s\in W^{1,r}(\D)$, and
since $F$ is locally smooth we get that
$w\in W^{1,r}_{loc}(\Omega)$; if $r=2$, it follows from Proposition~\ref{embexpS} that $e^s\in W^{1,q}(\Omega)$ for all $q\in[1,2)$, 
and thus $w=e^sF$ lies in $W^{1,q}_{loc}(\Omega)$. This proves $(iii)$.

Finally, if $\Omega$ is $C^1$-smooth and $w\not\equiv0$
(hence $w\neq0$ a.e.\ by the above argument), there exists a unique 
$s\in W^{1,r}(\Omega)$ satisfying the equations $\bar\partial s=\alpha\overline{w}/w$,
$\tr_{\partial\Omega}\mbox{\rm Re}(e^{i\theta_0}s)=\psi$,
$\int_{\partial\Omega} \mbox{\rm Im }(e^{i\theta_0}s)=\lambda$,
and \eqref{intdbar} yields \eqref{intdbarb}. 
Moreover, if  \eqref{decompexpH2} holds
for some $s\in W^{1,r}(\Omega)$
and some holomorphic $F$,
we find upon differentiating that  $\bar\partial s=\alpha\overline{w}/w$,
therefore factorization \eqref{decompexpH2} is unique with the 
aforementioned conditions. This proves $(ii)$.
\end{proof}

A weak converse to the similarity principle is as follows: if
$s \in W^{1,r}(\Omega)$ and $F$ is holomorphic on $\Omega$,
then $w=e^sF$ satisfies \eqref{eq:w} with 
$\alpha:=\bar \partial s\,e^{s}F/(e^{\bar s}\bar F)\in L^r(\Omega)$.
This remark shows that, in general, we cannot expect solutions of  
\eqref{eq:w} to lie in $L^\infty_{loc}(\Omega)$ when $r=2$.
\smallskip

\section{Holomorphic parametrization}
\label{x19}

When $r>2$, it follows from \cite[Theorem 3.13]{vekua} that for each   
holomorphic function 
$F$ on $\Omega$ and each $\alpha\in L^r(\Omega)$,
there is $\Phi\in W^{1,r}(\Omega)$ such that $w:=\Phi F$
satisfies \eqref{eq:w}. 
In this section we improve this assertion to
a strong converse of the similarity principle, valid for 
$2\leq r<\infty$,
which leads to a parametrization of pseudo-holomorphic functions by 
holomorphic functions.
We state the result for the disk, which is our focus
in the present paper, but we mention that it carries over at once
to Dini-smooth\footnote{A 
domain is Dini-smooth if its boundary has a parametrization with 
Dini-continuous derivative. Conformal maps between such domains have 
derivatives that extend continuously up to the boundary.}
simply 
connected  domains, granted the conformal invariance of
equation \eqref{eq:w} pointed out in \cite[Section 3.2]{BFL}.

\begin{thm} 
\label{BNpar}
Let $\alpha\in L^r(\D)$ for some $r\in[2,\infty)$, and let $F\not\equiv0$ be holomorphic on $\D$. 
Choose $\psi\in W_\R^{1-1/r,r}(\T)$, and $\lambda\in\R$.
Then there exists a unique $s\in W^{1,r}(\D)$  such that $w=e^s F$ is a solution of \eqref{eq:w}
with $\tr_{\T}\mbox{\rm Im}\,s=\psi$ and $\int_{\T} \mbox{\rm Re}\,s=\lambda$.
Moreover, \eqref{intdbarb} holds with some $C$ depending only on $r$.
\end{thm}

From the proof of the theorem, we obtain also the following variant thereof.

\begin{cor}
\label{BNreal}
Theorem~\ref{BNpar} remains valid if, instead of 
$\tr_{\T}\mbox{\rm Im}\,s=\psi$ and $\int_{\T} \mbox{\rm Re}\,s=\lambda$, 
we prescribe 
 $\tr_{\T}\mbox{\rm Re}\,s=\psi$
and $\int_{\T} \mbox{\rm Im}\,s=\lambda$.
\end{cor}

Before establishing  Theorem~\ref{BNpar},  we need to take a closer look
at pairs $s,F$ for which
\eqref{decompexpH2} and \eqref{eq:w} hold. We do this in the following
subsection.

\subsection{Arguments of pseudo-holomorphic functions}
Let $w\in L^\gamma_{loc}(\D)$ satisfy \eqref{eq:w}, $\gamma>r/(r-1)$,
and consider  factorization
\eqref{decompexpH2} provided by Lemma~\ref{expsf2}.
Locally around points where $F$ does not vanish, $w$ 
has a Sobolev-smooth  argument,
unique modulo $2\pi\ZZ$, which is given by $\arg w=\arg F+\mbox{\rm Im}\,s$.
Since $\log F$ is harmonic and $\bar \d  s=\alpha \bar w/w$, 
we deduce that around such points 
$\Delta \log w=4\d(\alpha e^{-2i\arg w})$. In particular, 
$\arg w$ satisfies 
the nonlinear (yet quasilinear) equation 
$\Delta\arg w=4\,{\rm Im}(\d(\alpha e^{-2i\arg w}))$, and then $\log|w|$ is  
determined by $\arg w$ up to a harmonic function that turns out to be
completely determined by \eqref{eq:w}. The lemma below dwells 
on this observation but avoids speaking of $\arg F$ 
(which may not be globally defined if $F$ has zeros).

\begin{lem} 
\label{arglisse}
Let $\alpha\in L^r(\D)$ for some $r\in[2,\infty)$ and let $F$ be a non 
identically zero holomorphic function in $\D$. If we set
$\beta:=\alpha\overline{F}/F$, then a function 
$s\in W^{1,r}(\D)$ is such that $w:=e^sF$ satisfies \eqref{eq:w} if
and only if $\bar\d s=\beta e^{-2i\text{\rm Im}\,s}$.
This is equivalent to saying that
$s=\varphi_1+i\varphi_2$ where
$\varphi_1,\varphi_2\in W^{1,r}_{\R}(\D)$ satisfy the relations 
\begin{align}
\label{eq:phi2}
\Delta\varphi_2&=4\,{\rm Im}\bigl(\d\left(\beta e^{-2i\varphi_2}\right)
\bigr),\\
\label{eq:phi1}
\varphi_1&={\rm Re}\,\mathcal{C}\left(\beta e^{-2i\varphi_2}\right)+v,
\end{align}
where $v$ is a harmonic conjugate to
the harmonic function $u\in W_\R^{1,r}(\D)$ such that 
$\tr_\T u =\tr_\T \,{\rm Im}\left(\mathcal{C}(\beta e^{-2i\varphi_2})\right)-
\tr_\T\varphi_2$.
\end{lem}

\begin{proof}
Using Proposition~\ref{embexpS}  to justify the computation
in case $r=2$, we 
find that $s\in W^{1,r}(\D)$ with $w=e^sF$ satisfies \eqref{eq:w} 
if and only if $\bar\d s-\beta e^{\bar s-s} =0$.
With the notation $\varphi_1:={\rm Re}\,s$  and  $\varphi_2:={\rm Im}\,s$ 
this is equivalent to
\begin{equation}
\label{eqs2}
\bar\d \varphi_1=\beta \exp\left(-2i\varphi_2\right)-i\bar\d\varphi_2,
\qquad \varphi_1,\varphi_2\in W^{1,r}_\R(\D). 
\end{equation}
Solving this $\bar\d$-equation for $\varphi_1$ using
the Cauchy operator, we can rewrite \eqref{eqs2} as 
\begin{equation}
 \varphi_1=\mathcal{C}\left(\beta e^{-2i\varphi_2}\right)-
i\varphi_2+A,
\qquad \varphi_1,\varphi_2\in W^{1,r}_\R(\D),
\label{resRe}
\end{equation}
where $A$ is holomorphic in $\D$. Since 
$\beta e^{-2i\varphi_2}\in L^r(\D)$  we obtain that
$\mathcal{C}(\beta e^{-2i\varphi_2})\in W^{1,r}(\D)$, hence
$\varphi_1, \,\varphi_2$ belong to $W^{1,r}(\D)$ 
if and only if $A$ does. 
Therefore, given $\varphi_2\in W^{1,r}_{\R}(\D)$,
equation  \eqref{resRe} gives rise to a real-valued $\varphi_1$ in $W^{1,r}(\D)$ if and only if the holomorphic function
$A$ lies in $W^{1,r}(\D)$ and satisfies the relation 
\begin{equation}
\label{resIm}
-{\rm Im}\,\mathcal{C}\left(\beta e^{-2i\varphi_2}\right)+
\varphi_2=
{\rm Im}\,A.
\end{equation}   
By the discussion after \eqref{intdbar} such an $A$ exists if and only 
if 
the left hand side of \eqref{resIm} is harmonic; since $\Delta$ 
commutes
with taking the imaginary part, this condition amounts to
$$
\Delta\varphi_2-4{\rm Im}\Bigl(\d\bar\d\mathcal{C}(\beta 
e^{-2i\varphi_2})\Bigr)=\Delta\varphi_2-4{\rm Im}\Bigl(
\d(\beta e^{-2i\varphi_2})\Bigr)=0
$$
which is \eqref{eq:phi2}. Then, by
\eqref{resIm},  
${\rm Im}\,A$ is the 
harmonic function $h\in W_\R^{1,r}(\D)$ having trace
$\tr_\T\varphi_2-\tr_\T{\rm Im}\,\mathcal{C}(\beta e^{-2i\varphi_2})\in 
W^{1-1/r,r}(\T)$. Subsequently ${\rm Re}\,A={\rm Im}\,(iA)$ must be 
a harmonic conjugate to $-h=u$, and taking real parts in 
\eqref{resRe} yields \eqref{eq:phi1}.
\end{proof}

\subsection{Proof of Theorem \ref{BNpar}}

\subsubsection{Existence part.}
\label{existence}
In this subsection, we prove existence of $s$ in the conditions of Theorem~\ref{BNpar}.
Note that \eqref{intdbarb} will automatically 
hold by Lemma~\ref{expsf2}~(ii) applied with $\theta_0=-\pi/2$.
Let $A\in W^{1,r}(\D)$ be  holomorphic in $\D$ with
$\tr_\T \mbox{\rm Re\,}A=\psi$ and $\int_\T{\rm Im\,}A=-\lambda$.
Writing $e^sF=e^{s-iA}(e^{iA}F)$, we see that we may  assume
$\psi=0$ and $\lambda=0$ upon replacing $F$ by $e^{iA} F$.
In addition, upon changing
$\alpha$ by $\alpha {\overline{F}/ F}$, 
we can further suppose that $F\equiv1$
thanks to \eqref{eq:phi2} and \eqref{eq:phi1}.

We first deal with the case  $r=2$ and begin  with 
fairly smooth $\alpha$, say $\alpha\in W^{1,2}(\D)\cap L^\infty(\D)$. 
Consider the following (non-linear) operator $G_\alpha$ acting on 
$\varphi\in W_\R^{1,2}(\D)$:
\begin{equation}
\label{defGalpha}
G_\alpha(\varphi)(z):=-\frac2\pi \int_\D \log\Bigl| \frac{1-\bar z t}{z-t}
\Bigr|\mbox{\rm Im}\Bigl(\d\bigl(\alpha(t)  e^{-2i\varphi(t)}\bigr)\Bigr) 
\,dm(t),\quad z\in\D.
\end{equation}
Since $|e^{-2i\varphi}|=1$ and $\alpha\in W^{1,2}(\D)\cap L^\infty(\D)$, 
we get from Proposition~\ref{embexpS} 
that $\partial(\alpha e^{-2i\varphi})\in L^2(\D)$, therefore
the above integral exists for every $z\in\C$ by
the Schwarz inequality. In fact,  $G_\alpha(\varphi)$ is  
the Green potential of 
$4\mbox{\rm Im}\bigl(\d\bigl(\alpha e^{-2i\varphi}\bigr)\bigr)$ 
in $\D$, that is, its distributional  Laplacian is 
$4 \mbox{\rm Im}\bigl(\d\bigl(\alpha  e^{-2i\varphi}\bigr)\bigr)$ and its 
value on $\T$ is zero, compare to \cite[Section 4.8.3]{aim}. 
To prove existence of $s$ subject to the conditions $\psi=0$, $\lambda=0$, 
and $F\equiv1$,  
it suffices by Lemma~\ref{arglisse}
to verify that $G_\alpha$ has a 
fixed point in  $W_\R^{1,2}(\D)$. 
First, we check that $G_\alpha$ is compact from
$W_\R^{1,2}(\D)$ into itself,
meaning that it is continuous and maps bounded sets to relatively compact ones.

\begin{lem}
\label{Galphacomp}
If $\alpha\in W^{1,2}(\D)\cap L^\infty(\D)$, then the operator 
$G_\alpha$ is bounded and continuous from $W_\R^{1,2}(\D)$ into 
$W_\R^{2,2}(\D)$ and it is compact from $W_\R^{1,2}(\D)$ into itself.
\end{lem}

\begin{proof}
To prove the boundedness and continuity of
$G_\alpha$ from  $W_\R^{1,2}(\D)$ into $W_\R^{2,2}(\D)$, observe from 
\eqref{CR2} and the dominated convergence theorem that
the map $\varphi\mapsto  \mbox{\rm Im}\, \left(\d\left(\alpha  e^{-2i\varphi}\right)\right)$ is bounded and continuous from 
$W^{1,2}_\R(\D)$ into $L^2_\R(\D)$. Therefore it suffices to prove 
the boundedness from
$L^2_\R(\D)$ into $W^{2,2}_\R(\D)$ of the linear potential
operator:
$$
P(\psi):= -\frac{1}{2\pi}\int_\D \log\Bigl|\frac{1-\bar z t}{z-t}\Bigr|\psi(t) \,dm(t).
$$
The latter is a consequence of properties of the
Cauchy and Beurling transforms listed in Section 
\ref{sec:notations_generales}  \cite[Section 4.8.3]{aim}. Compactness of $G_\alpha$ from $L^2(\D)$ into
$W^{1,2}(\D)$ now follows from compactness of the embedding of
$W^{2,2}(\D)$ into
$W^{1,2}(\D)$
asserted by the Rellich--Kondrachov theorem.
\end{proof}

Since $G_\alpha$ is compact on $W_\R^{1,2}(\D)$, a sufficient 
condition for it to have a fixed point is given by
the Leray--Schauder theorem \cite[Theorem 11.3]{gt}:
there is a number $M$ for which the {\it a priori} estimate 
$\|\varphi \|_{W^{1,2}(\D)}\le M$ holds
whenever $\varphi\in W_\R^{1,2}(\D)$ and $\varepsilon\in [0,1]$ 
satisfy
\begin{equation}
\label{pfeps}
\varphi=\varepsilon G_\alpha(\varphi).
\end{equation}
Now, if \eqref{pfeps} is true, then \eqref{eq:phi2} is satisfied with 
$\beta=\varepsilon\alpha$ and $\varphi$ instead of $\varphi_2$.
Therefore by Lemma~\ref{arglisse}, there exist 
$\varphi_{1,\varepsilon}\in W_\R^{1,2}(\D)$ and 
$s_\varepsilon:=\varphi_{1,\varepsilon}+i\varphi$ such that
$$
\bar{\d}e^{s_\varepsilon}=\varepsilon \alpha\,\overline{e^{s_\varepsilon}}.
$$
Applying Lemma~\ref{expsf2}~(ii) with $\Omega=\D$, $F\equiv1$,
$s=s_\varepsilon$,
$\psi\equiv0$,
$\theta_0=-\pi/2$ and $\lambda=0$, we get from \eqref{intdbarb} that
for some absolute constant $C$
$$
\|\varphi\|_{W^{1,2}(\D)}\leq \|s_\varepsilon\|_{W^{1,2}(\D)}\le \varepsilon C\|\alpha\|_{L^2(\D)}
\leq C\|\alpha\|_{L^2(\D)}=:M.
$$
Thus,  $G_\alpha$ indeed has a fixed point,  which settles the case
$r=2$ and $\alpha\in W^{1,2}(\D)\cap L^\infty(\D)$. 

Next, we relax our restriction on $\alpha$ and assume only 
that it belongs to $L^2(\D)$. 
Let $(\alpha_n)$ be a sequence in $\mathcal{D}(\D)$ that
converges to $\alpha$ in $L^2(\D)$. 
By the first part of the proof, there is a sequence 
$(s_n)\subset W^{1,2}(\D)$ such that $\mbox{\rm Im}\, \tr_{\T}s_n=0$ and
$\int_\T \mbox{\rm Re}\,\tr_{\mathbb T}\, s_n=0$,
satisfying 
$\overline{\d}e^{s_n}=\alpha_n\overline{e^{s_n}}$ as well as
({\it cf.} \eqref{intdbarb}) 
\begin{equation}
\label{limsn}
\|s_n\|_{W^{1,2}(\D)}\le C\|\alpha_n\|_{L^2(\D)}\leq C'.
\end{equation}
By the Rellich--Kondrachov theorem we can find a subsequence, again denoted 
by $(s_n)$, converging pointwise and in all $L^q(\D)$, 
$1\le q<\infty$ to some function $s$.
By dominated convergence, 
the functions $\bar\partial s_n=\alpha_n{e^{-2i{\rm Im}\, s_n}}$ 
converge to 
$\alpha{e^{-2i{\rm Im}\, s}}$ in $L^2(\D)$. Thus, applying
\eqref{intdbar} with $A=s_n-s_m$, 
$a=\bar\partial s_n-\bar\partial s_m$, 
$\theta_0=-\pi/2$, $\psi\equiv0$, and $\lambda=0$, 
we conclude that $(s_n)$ is a Cauchy sequence in $W^{1,2}(\D)$  
which must therefore converge to $s$. 
Hence  $s\in W^{1,2}(\D)$,
$\mbox{\rm Im}\, \tr_{\T}s=0$, $\int_\T \mbox{\rm Re}\,s=0$, and 
$\bar\partial s=\alpha e^{-2i{\rm Im}\,s}$.
By Lemma~\ref{arglisse}, this establishes existence of $s$ when $r=2$.
Suppose finally that $\alpha\in L^r(\D)$ for some $r>2$. 
{\it A fortiori} $\alpha\in L^2(\D)$,  so by what precedes
there is $s\in W^{1,2}(\D)$ 
such that $\mbox{\rm Im}\, \tr_{\T}s=0$, $\int_\T \mbox{\rm Re}\,s=0$, and
$\overline{\d}e^{s}=\alpha\overline{e^{s}}$. To see that in fact
 $s\in W^{1,r}(\D)$, we apply Proposition~\ref{embexpS} to get 
$\bar\partial s=\alpha e^{-2i{\rm Im}\,s}=:a\in L^r(\D)$.
Then, equation \eqref{intdbar} implies that $s$ is the unique 
function $A\in W^{1,r}(\D)$ satisfying 
$\mbox{\rm Im}\, \tr_{\T}A=0$, $\int_\T \mbox{\rm Re}\,A=0$, and
$\overline{\d}A=a$.\hfill$\square$

\subsubsection{Uniqueness part}
\label{uniqueness}
In this subsection we establish uniqueness of $s$ in the conditions of Theorem~\ref{BNpar}.
Clearly, it is enough to consider $r=2$.  
Consider two functions $w_1=e^{s_1}F$ and $w_2=e^{s_2}F$ meeting
\eqref{eq:w} on $\D$ with $s_j\in W^{1,2}(\D)$,
$\tr_\T{\rm Im}\,s_j=\psi$, and  $\int_\T {\rm Re}\,s_j=\lambda$ for $j=1,2$.
We define
$$
s(z):=s_1(z)-s_2(z)\in W^{1,2}(\D)
$$
and we  must prove that $s\equiv0$.  First we 
estimate the $\bar\partial$-derivative of $s$:

\begin{lem}
\label{lem1S}
There is a constant $C>0$ such that, for a.e.  $z\in\D$, we have
\begin{equation}
\label{Gronc}
|\bar\d s(z)|\leq C|{\rm Im}~ s(z)|\, |\alpha(z)|. 
\end{equation}
\end{lem}

\begin{proof} Setting $\beta:=\alpha \bar F/F$ and 
using again Lemma~\ref{arglisse}, we find that
$\bar \d s_j=\beta e^{-2i\,\text{Im}\,s_j}$. Hence, 
$\bar\d s=\beta e^{-2i\,\text{Im}\,s_1}
\bigl(1-e^{2i\,{\rm Im}s}\bigr)$, and \eqref{Gronc} follows at once.
\end{proof}

Next, we extend the function $s$ outside of  
$\overline{\D}$ by reflection: 
\begin{equation}
\label{defsref}
s(z):=\overline{s\left({1/ \bar z}\right)},\qquad z\in\C\setminus\overline{\D},
\end{equation}
Observe that since $s$ is real-valued on $\T$, this extension 
makes $s\in W^{1,2}_{loc}(\C)$, see, for example \cite[Theorem 2.54]{Demengel}.

\begin{lem}
\label{lem2S}
There is a constant $C>0$ such that, for a.e. $z\in\C\setminus\overline{\D}$,
\begin{equation}
\label{ineglogders}
|\bar\d s(z)|\le 
C \frac{\bigl|{\rm Im}\,s(z)\bigr|}{|z|^2}\,|\alpha(1/\bar z)|. 
\end{equation}
\end{lem}

\begin{proof}
Putting $\zeta={1/\bar z}=:U(z)$ and applying the chain rule, we get
(since $\partial f=0$) that
$$
\d\left(s\left({1/\bar z}\right)\right)=
\left((\d_{\bar\zeta}~s)\circ U\right)\d\bar U=
-\frac{1}{z^2}\bar\d s\left({1/\bar z}\right).
$$
Thus, 
$$
\bar\d s(z)=\overline{\d\left(s\left({1/\bar z}\right)\right)}=-\overline{\Bigl(\frac{\bar\d s({1/\bar z})}{z^2}\Bigr)},
\qquad |z|>1,
$$
and applying Lemma~\ref{lem1S} gives  \eqref{ineglogders} in view of
\eqref{defsref}.
\end{proof}

From the two previous lemmas we derive the inequality: 
\begin{equation}
\label{doms}
|\bar\d s(z)|\leq C\frac{|\text{Im}\,s(z)|}{1+|z|^2}\,|\alpha(Q(z))|,\qquad
a.e.\ z\in \C,
\end{equation}
where $Q(z)$ is equal to $z$ if $|z|\le 1$ and to $1/\bar z$ otherwise.
Since $\alpha\in L^2(\D)$, it follows from  \eqref{doms} and the change of 
variable formula that
$\bar \partial s/s\in L^2(\C)$. 

Recall now definition \eqref{defCauchy2}. We 
introduce two auxiliary functions $\psi,\phi$ on $\C$:
\begin{equation}
\label{defpsiphi}
\psi:=\mathcal{C}_2\bigl(\bar\partial s/s\bigr),\qquad
\phi:=\exp(-\psi).
\end{equation}

Since $\bar\d s/s\in L^2(\C)$, we know that
$\psi\in W^{1,2}_{loc}(\C)$ with $\bar\partial\psi=\bar\d s/s$.
Consider  the function $s\phi$ on $\C$. By
Proposition~\ref{embexpS},
we compute from \eqref{defpsiphi} using the Leibniz rule that
$\bar\partial(s\phi)=0$, hence $s\phi$ is an entire function.
We claim that
\begin{equation}
\liminf_{R\to+\infty}\Bigl(\frac 1 R\int_{\T_R}\log^+|s\phi(\xi)|\,|d\xi|-\frac 12\log R\Bigr)<0.
\label{liu}
\end{equation}
Indeed, taking into account \eqref{defsref} and the fact that
$s_{|\D}\in L^\ell(\D)$ for all  $1\le \ell<\infty$ by the Sobolev 
embedding theorem, we get  from Jensen's inequality  upon choosing
$\ell>48\pi$ that
\begin{multline}
\label{Jensen}
\frac{1}{R^2}\int_{R<\rho<2R}\int_{0<\theta<2\pi}\log^+|s(\rho e^{i\theta})|\,
\rho\,d\rho\,d\theta
\\ \le \frac {12\pi}\ell\frac{4R^2}{3\pi}\int_{1/(2R)<\rho<1/R}\int_{0<\theta<2\pi}
\log^+(|s(\rho e^{i\theta})|^\ell)\,\rho\,d\rho\,d\theta
\\ \le
\frac{12\pi}\ell\log\Bigl[\frac{4R^2}{3\pi}\int_{1/(2R)<\rho<1/R}\int_{0<\theta<2\pi}
\max\{1,|s(\rho e^{i\theta})|^\ell\}\,\rho\,d\rho\,d\theta\Bigr]
\\ \le \frac{1}{\delta}\log R+C
\end{multline}
for some $\delta>2$ and  some $C>0$,  whenever $R\ge 1$.

In another connection, it follows from \eqref{estdRC2} and the Schwarz 
inequality that
\begin{multline}
\label{estcourpsi}
\frac{1}{\pi R^2}\int_{R<\rho<2R}\int_{0<\theta<2\pi}|\psi(\rho e^{i\theta})|\,
\rho\,d\rho\,d\theta\leq 
\frac{\|\psi\|_{L^2(\D_{2R})}}{\sqrt{\pi} R}\\=
O\left((\log R)^{1/2}\right),
\quad R\to+\infty.
\end{multline}
Since $\log^+|s\phi|\le \log^+|s|+|\psi|$, claim \eqref{liu} 
easily follows 
from \eqref{Jensen} and \eqref{estcourpsi}.

Since $\log |s\phi|$ is subharmonic on $\C$, for 
$|z|<R$ we have 
\begin{multline*}
2\pi \log|s\phi(z)|\leq\frac{R+|z|}{R-|z|}
\int_0^{2\pi}\log^+\left|s\phi(Re^{i\theta})\right|\,d\theta
\\ \le\frac{R+|z|}{R-|z|}\,\frac{1}{R}\int_{\T_R}\log^+|s\phi(\xi)|\,|d\xi|
\end{multline*}
(see \cite[Theorem 2.4.1]{Ransford}), so by \eqref{liu}
there is a sequence $\rho_n\to+\infty$  for which 
$\sup_{\T_{\rho_n}} |s\phi|=O(\rho_n^{1/2})$.
Therefore, by an easy modification of  Liouville's theorem, $s\phi$ must be
 a constant.

More generally, \eqref{doms} remains valid if we replace $s$ by $s-a$ 
for $a\in\R$, entailing that
\begin{equation}
\label{domsa}
\left|\frac{\bar\d s(z)}{s(z)-a}\right|\leq C
\left|\frac{\alpha(Q(z))}{ 1+|z|^2}\right|\in L^2(\C),\qquad
a.e.\ z\in \C,\quad a\in\R.
\end{equation}
Thus, reasoning as before, we deduce that 
there is a complex-valued function $b$ such that
\begin{equation}
\label{x12}
(s(z)-a)\phi_a(z)\equiv b(a),\qquad a\in\R,
\end{equation}
with
$$
\psi_a:=\mathcal{C}_2\left(\bar\d s/(s-a)\right),\quad \phi_a:=\exp(-\psi_a).
$$
Fix $R>1$. By \eqref{domsa}, 
Corollary~\ref{equicontC2}, and 
Proposition~\ref{embexpS}, 
the sets 
$\{\phi_a|_{\D_R}\}_{a\in\R}$ and $\{\phi^{-1}_a|_{\D_R}\}_{a\in\R}$ are 
bounded in 
$W^{1,q}(\D_R)$ for $q\in[1,2)$, hence also in $L^2(\T)$
by the trace and the Sobolev embedding theorems. 
Fix $A>0$ such that $\Lambda(\{\xi\in\T:|s(\xi)|\le A\})=\lambda>0$.
For  each $\delta>0$ we can cover the interval $[-A,A]$ by $N\leq A/\delta+1$
open intervals of length $2\delta$, hence there exists 
$a=a(\delta)\in[-A,A]$ with
$\Lambda(E_a)\ge \lambda \delta/(A+\delta)$, where
$E_a=\{\xi\in\T:|s(\xi)-a|\le \delta\}$. (We use here that $s$ is real-valued on $\T$.)
Moreover, we observe from \eqref{x12} and the Schwarz inequality that
$$
|b(a)|\Lambda(E_a)= \int_{E_a}|s(\xi)-a||\phi_a(\xi)|\,d\Lambda(\xi)\le 
\delta\|\phi_a\|_{L^2(\T)} \Lambda(E_a)^{1/2}.
$$ 
This lower bound on $\Lambda(E_a)$ now gives us that 
$$
|b(a)|\le\delta^{1/2}\sqrt{(A+\delta)/\lambda}\,\sup_{|a|\le A}\|\phi_a\|_{L^2(\T)},
$$
implying that 
$b(a(\delta))\to 0$ as $\delta\to 0$. By compactness, we 
can pick a sequence $\delta_n\to0$ 
such that
$a_n:=a(\delta_n)\to c\in[-A,A]$. Considering the equalities
$s-a_n=b(a_n)\phi^{-1}_{a_n}$
and taking into account the boundedness of $\{\phi^{-1}_{a_n}\}$ in 
$L^2(\D_R)$, we find  that $s\equiv c$ on $\D_R$. Since $R$ is arbitrary, 
$s$ is constant on $\C$, and actually $s\equiv0$ because $\int_\T s=0$. 
\hfill$\square$

A similar argument gives the following result which seems to be of independent interest.

\begin{thm}
\label{x16}
If 
$s\in W^{1,2}_{loc}(\C)$ satisfies 
$$
|\bar\d s(z)|\le |{\rm Im}\,s(z)| \,g(z)
$$
for some non-negative function  $g\in L^2(\C)$, and if  
$$
\int_{\C\setminus\D}\frac{|s(\xi)|^\ell\,{\rm d}\xi\wedge{\rm d}\bar\xi}{|\xi|^4}<\infty,
$$
for some $\ell>48\pi$,
then 
${\rm Im}\,s$ is of constant sign a.e. in $\C$.
\end{thm}

The example $s(z)=i+(1+|z|)^{-\beta}$, $\beta>0$ shows that, under these conditions, $s$ is not necessarily a constant. On the other hand, the value $48\pi$ 
is not necessarily sharp.

It is interesting to compare this result to known 
Liouville-type theorems like 
\cite[Proposition 3.3]{ap} and \cite[Theorem 8.5.1]{aim}.

\begin{proof}[Sketch of proof] 
For any real $d$, \eqref{x12} gives us for small $\delta>0$ that 
$m\{\xi\in\D:|{\rm Im}\,s(\xi)|<\delta,\,|{\rm Re}\,s(\xi)-d|<\delta\}\le c \delta^5/\inf_{a\in\R}|b(a)|^5$ with $c$ independent of $d$; 
therefore,   
$m\{\xi\in\D:|{\rm Im}\,s(\xi)|<\delta,\,|{\rm Re}\,s(\xi)|<1/\delta\}\le c \delta^3/\inf_{a\in\R}|b(a)|^5$. 
Since $s\in W^{1,2}_{loc}(\C)$, by the John--Nirenberg theorem we have
$m\{\xi\in\D:|s(\xi)|>1/\delta\}\le c \delta^3$. Finally, if ${\rm Im}\,s$ changes sign in $\D$, then by the H\"older inequality we obtain 
$m\{\xi\in\D:|{\rm Im}\,s(\xi)|<\delta\}\ge c \delta^2$. As a result, passing to the limit $\delta\to 0$, we obtain that $\inf_{a\in\R}|b(a)|=0$.
\end{proof}

\subsection{Proof of Corollary \ref{BNreal}}
\label{preuvecorim}
Uniqueness of $s$ is established as in Theorem~\ref{BNpar}, except that the 
right hand side of \eqref{defsref} now has a minus sign because $s$ is pure 
imaginary on $\T$. Note also
that \eqref{intdbarb} holds by Lemma~\ref{expsf2}~(ii) applied with $\theta_0=0$.

Passing to existence of $s$, the argument given early in subsection 
\ref{existence} applies with obvious 
modifications to show that we may assume $\psi=0$, $\lambda=0$ and $F\equiv1$.
Moreover, it is enough to  prove the result when $r=2$ and 
$\alpha\in \mathcal{D}(\D)$, for then the passage to $\alpha\in L^2(\D)$ and, subsequently, to $r>2$ is like in the theorem. 

So, let us put $r=2$, fix $\alpha\in \mathcal{D}(\D)$, and write
$s(\psi,\lambda,F)$ to emphasize the dependance on $\psi$, $\lambda$ and $F$ 
of the function $s\in W^{1,2}(\D)$ whose existence and uniqueness 
is asserted by Theorem~\ref{BNpar}. For $u\in W^{1/2,2}_{\R}(\T)$, we
denote by $E(u)\in W^{1,2}_\R(\D)$ the harmonic extension of $u$, 
{\it i.e.} $E(u)$ is harmonic and $\tr_\T\,E(u)=u$. We put
$H(u)\in W^{1,2}(\D)$ for the holomorphic function such that
$\text{Im}\,H(u)=E(u)$ and $\int_\T \text{Re}\,H(u)=0$.
Observe that $\alpha \exp(-2i E(u))$ lies in $W^{1,2}(\D)\cap L^\infty(\D)$; therefore, the
operator $G_{\alpha e^{-2i E(u)}}$ defined by \eqref{defGalpha} is 
compact from $W^{1,2}_\R(\D)$ into itself by Lemma~\ref{Galphacomp}.
In the course of the proof of Theorem \ref{BNpar}, we showed  that 
it has a unique  fixed point which is none but 
$\text{Im}\,s(0,0,e^{H(u)})=:\mathcal{F}(u)$.
Furthermore, by \eqref{intdbarb}, we have  
$\|\mathcal{F}(u)\|_{W^{1,2}(\D)}\leq C\|\alpha\|_{L^2(\D)}$ 
for some absolute constant $C$. 

\begin{lem}
\label{contFP}
The {\rm(}nonlinear{\rm)} operator 
$u\mapsto \mathcal{F}(u)$ is compact from \newline\noindent $W^{1/2,2}_{\R}(\T)$ into 
$W^{1,2}_{\R}(\D)$.
\end{lem}

\begin{proof} Pick 
a sequence $(u_n)$ converging to $u$ in $W^{1/2,2}(\T)$. 
By elliptic regularity, $H(u_n)$ converges to $H(u)$ in $W^{1,2}(\D)$, 
and in particular
$\|H(u_n)+\mathcal{F}(u_n)\|_{W^{1,2}(\D)}$ 
is bounded independently of $n$.
Besides,  
by \eqref{defGalpha} and the definition of $\mathcal{F}$, we  see that
\begin{equation}
\label{shiftmultarg}
\mathcal{F}(u_n)=G_{\alpha e^{-2i E(u_n)}}(\mathcal{F}(u_n))
=G_{\alpha}(E(u_n)+\mathcal{F}(u_n));
\end{equation}
hence, Lemma~\ref{Galphacomp} implies that  $(\mathcal{F}(u_n))_{n\in\NN}$ is
relatively compact in $W^{1,2}_\R(\D)$. 
Let some subsequence, again denoted by $(\mathcal{F}(u_n))$,
converges to $\varphi$ in $W^{1,2}_\R(\D)$. 
Then $(E(u_n)+\mathcal{F}(u_n))$ converges to $E(u)+\varphi$ in  
$W^{1,2}_\R(\D)$,
therefore by \eqref{shiftmultarg} and the continuity of $G_\alpha$ we obtain 
that $\varphi=G_{\alpha e^{-2i E(u)}}(\varphi)$. This means that 
$\varphi=\mathcal{F}(u)$, hence the latter 
is the only limit point of  $(\mathcal{F}(u_n))_{n\in\NN}$, 
which proves the continuity of $\mathcal{F}$.

If we assume only that $\|u_n\|_{W^{1/2,2}(\T)}$ is bounded independently 
of $n$, then elliptic regularity still gives us that
$\|E(u_n)\|_{W^{1,2}(\D)}$ is bounded, hence
$(E(u_n)+\mathcal{F}(u_n))_{n\in\NN}$ is again bounded in $W^{1,2}_\R(\D)$.
As before it follows that 
$(\mathcal{F}(u_n))_{n\in\NN}$ is relatively compact in
$W^{1,2}_{\R}(\D)$, as desired.
\end{proof}

Given $u\in W^{1/2,2}_\R(\T)$,  let
$\tilde{u}:=-\tr_\T \text{Re}\,H(u)$ denote the so called
conjugate function of $u$. That is, $\tilde{u}$ is the trace 
of the harmonic conjugate of $E(u)$ that has zero mean on $\T$.
Put
$\mathcal{M}\subset W^{1/2,2}_\R(\T)$ for the subspace of functions 
with zero mean.  By \eqref{intdbar}, the map $u\mapsto \tilde{u}$ is 
continuous from $W^{1/2,2}_\R(\T)$ into $\mathcal{M}$, and since 
$\tilde{\tilde{u}}=-u+\int_\T u$, 
it is a homeomorphism of $\mathcal{M}$.
Pick $u\in \mathcal{M}$ and let $\varphi:=\text{Im}\,s(u,0,1)$.
Since $s(u,0,1)-H(u)=s(0,0,e^{H(u)})$ we have 
$\varphi=E(u)+\mathcal{F}(u)$. Set for simplicity
$R(u):=\mathcal{C}\bigl(\alpha \exp\{-2i(E(u)+\mathcal{F}(u))\}\bigr)$.
Applying the trace and conjugate operators to \eqref{eq:phi1}, we 
see that $\tr_\T\text{Re}\,s(u,0,1)=0$ 
if and only if 
\begin{equation}
\label{condRe0}
u=tr_\T\,{\rm Im}\Bigl(R(u)-\int_\T
R(u)\Bigr)\,-\,
\widetilde{\overbrace{\tr_\T\,{\rm Re}\Bigl(R(u)-\int_\T
R(u)\Bigr)}}.
\end{equation}
Let $B(u)$ denote the right hand side of \eqref{condRe0}. To complete the 
proof, it remains to show that the (nonlinear) operator $B$ has a fixed point $u_0$ 
in $\mathcal{M}$. Then the function $s(u_0,0,1)$ would satisfy the conditions of the corollary.

To prove existence of a fixed point, we claim first that $B$ is compact 
from $\mathcal{M}$ into itself. Indeed, 
by elliptic regularity, 
$E$ is linear and bounded from  $W^{1/2,2}_\R(\T)$ into  $W^{1,2}_\R(\D)$
while $\mathcal{F}$ is  compact by  Lemma~\ref{contFP}. A fortiori,
$E+\mathcal{F}$ is 
bounded and continuous from $\mathcal{M}$ into  $W^{1,2}_\R(\D)$.
Moreover, as
$\alpha\in \mathcal{D}(\D)$, it follows from
\eqref{CR2} and the dominated convergence theorem that 
$h\mapsto \alpha\exp(-2i h)$ is bounded and 
continuous from $W^{1,2}_\R(\D)$ into 
$W^{1,2}_0(\D)$. In addition, we get from \eqref{CauchyW12} that
$\mathcal{C}$ is bounded and linear  from $W^{1,2}_0(\D)$ into $W^{2,2}(\D)$, 
hence compact into $W^{1,2}(\D)$ by the Rellich--Kondrachov theorem. Finally, 
by the trace theorem, $g\mapsto \tr_\T\text{Im}\,(g-\int_\T g)$ is linear and 
bounded from $W^{1,2}(\D)$ into $\mathcal{M}$. Since the conjugate operator is
linear and bounded on $\mathcal{M}$ and composition 
with bounded continuous maps preserves compactness, the claim follows.
Appealing now  to the Leray--Schauder theorem, we know that $B$ has a fixed 
point if we can  find a constant $M$ such that
$\|u \|_{W^{1/2,2}(\T)}\le M$ holds
whenever $u=\varepsilon B(u)$ for some $\varepsilon\in [0,1]$.
However, such a $u$ must be equal to $\text{Im}\,\tr_\T s_\varepsilon$,
where $s_\varepsilon\in W^{1,2}(\D)$ has pure imaginary trace with zero mean
on $\T$ and $e^{s_\varepsilon}$ satisfies \eqref{eq:w} with $\alpha$ 
replaced by $\varepsilon\alpha$. Thus, from Lemma~\ref{expsf2}~(ii) applied with $\theta_0=0$, we conclude that 
$M=C\|\alpha\|_{L^2(\D)}$ will do for some absolute constant $C$. 

\section{Hardy spaces on the disk}
\label{HSD}

\subsection{Holomorphic Hardy spaces}
\label{HH}
For $p\in[1,\infty)$, let $H^p=H^p(\D)$ be the Hardy space of holomorphic 
functions  $f$ on $\D$ with
\begin{equation} \label{esssupp}
\left\Vert f\right\Vert_{H^p}:= 
\sup_{0<\rho  < 1}
\left(\frac{1}{2\pi}\int_0^{2\pi} \left\vert
f(\rho e^{i\theta})\right\vert^p\,d\theta\right)^{1/p} <+\infty. 
\end{equation}
The space $H^\infty$ consists of bounded holomorphic functions endowed with
the $sup$ norm. 
We refer to  \cite{duren,gar} for the following standard facts on 
holomorphic Hardy spaces. 

Each $f\in H^p$ has a non-tangential limit at a.e. $\xi\in\T$, 
which is also the
$L^p(\T)$ limit of $f_\rho (\xi):=f(\rho \xi)$ as $\rho \to1^-$ and
whose norm matches the supremum in \eqref{esssupp}. Actually 
$\|f_\rho \|_{L^p(\T)}$ is non-decreasing with $\rho $, hence instead of
\eqref{esssupp} we could as well have set\footnote{ In fact
\eqref{esssupp} expresses that $|f|^p$ has a harmonic majorant whereas
\eqref{esssuppa} bounds the $L^p$-norm of $f$ on curves tending 
to the boundary;
the first condition defines the Hardy space and the second 
the so-called Smirnov space. These coincide when  harmonic 
measure and arclength are 
comparable   on the boundary, 
\cite[Chapter 10]{duren}, \cite{TuHa2}, which is 
the case for smooth domains. The name ``Hardy space'' is then more common.}
\begin{equation} \label{esssuppa}
\left\Vert f\right\Vert_{H^p}:= 
\sup_{0<\rho  < 1}
\Bigl(\int_{\T_\rho} \left\vert
f(\xi)\right\vert^p |d\xi|\Bigr)^{1/p} <+\infty 
\end{equation}
where the integral is now with respect to the arclength.
As usual, we keep the same notation for $f$ 
and its non-tangential limit when no confusion can arise,
or write sometimes $f_{|\T}$ to emphasize that the non-tangential 
limit lives on $\T$.
Note that  $f_{|\T}$  coincides with 
$\tr f$ when  $f\in W^{1,p}(\D)$ \cite{BLRR}. Each function in $H^p$ is 
both the Cauchy and the Poisson integral of its non-tangential limit.
As regards the non-tangential maximal function, for $1\leq p<\infty$
and $f\in H^p$ we have 
\begin{equation}
\label{ntborne}
\|{\mathcal M}_\gamma f\|_{L^p(\T)}\leq C\|f\|_{L^p(\T)},
\end{equation}
where the constant $C$ depends only on $\gamma$ and $p$ 
\cite[Chapter II, Theorem 3.1]{gar}.

Traces of $H^p$-functions on $\T$ are exactly those functions in $L^p(\T)$
whose Fourier coefficients of negative index do vanish. In particular, 
if $f\in H^p$ and $f_{|\T}\in L^q(\T)$, then $f\in H^q$. It is obvious from 
Fubini's theorem that  $H^p\subset L^p(\D)$, but actually one can affirm more:
\begin{equation}
\label{sum2p}
\|f\|_{L^\lambda(\D)}\leq C\|f\|_{H^p},\qquad p\leq \lambda<2p,
\end{equation}
where $C=C(p,\lambda)$; for a proof see \cite[Theorem 5.9]{duren}.
A sequence $(z_l)\subset\D$ is the zero set of a nonzero $H^p$  function,
taking into account the multiplicities, if and only if it satisfies
the \emph{Blaschke condition}:
\begin{equation}
\label{Blaschke}
\sum_l(1-|z_l|)<\infty.
\end{equation}
A non-negative function $h\in L^p(\T)$ is such that $h=|f_\T|$ for some
nonzero $f\in H^p$ if and only if $\log h\in L^1(\T)$.
This entails that a nonzero $H^p$ function cannot vanish 
on a subset of strictly positive Lebesgue measure on $\TT$.

For $1<p<\infty$ and for every $\psi\in L^p_\RR(\T)$
there exists $g\in H^p$
such that $\mbox{Re}\,g=\psi$ on $\T$ \cite[Chapter III]{gar}. Such a $g$ is
unique up to an additive pure imaginary constant,
and if we normalize it so that $\int_\T\text{Im}\,g=0$, then
$\|g\|_{H^p}\le C 
\|\psi\|_{L^p(\T)}$ with  $C=C(p)$. In fact $g=u+iv$ on $\D$, 
where $u$
is the Poisson integral of $\psi$ and $v$ is the Poisson integral of 
\begin{equation}
\widetilde{\psi}(e^{i\theta}) := \lim_{\varepsilon\to0}
\frac{1}{2 \, \pi} \, 
\int_{\varepsilon<|\theta-t|<\pi} \frac{\psi(e^{it})}{\tan(\frac{\theta-t}{2})}
\, dt
\label{RHt}
\end{equation}
which is the so-called \emph{conjugate function} of $\psi$. This definition
carries over to $L^p(\T)$ the conjugation operator 
 $\psi\mapsto\widetilde{\psi}$ already introduced on $W^{1/2,2}(\T)$ after 
the proof of Lemma~\ref{contFP}. 
It is a theorem of
M. Riesz that the conjugation operator 
maps $L^p(\T)$ continuously 
into itself. By elliptic regularity, it is also continuous from 
$W^{1-1/p,p}(\T)$ into itself.

When $\psi\in L^1(\T)$, the conjugate function $\widetilde{\psi}$ is still
defined pointwise almost everywhere {\it via} \eqref{RHt} but it does not necessarily  
belong to $L^1(\T)$.

For $p\in(1,\infty)$, a non-negative function $\mathfrak{w}\in L^1(\T)$ 
is said to satisfy the Muckenhoupt condition $A_p$ if 
\begin{equation}
\label{defAp}
\{\mathfrak{w}\}_{A_p}:=\sup_{I}\Bigl(\frac{1}{\Lambda(I)}\int_{I} \mathfrak{w}\,d\Lambda\Bigr)
\Bigl(\frac{1}{\Lambda(I)}\int_I \mathfrak{w}^{-1/(p-1)} d\Lambda\Bigr)^{p-1}<+\infty,
\end{equation}
where the supremum is taken over all arcs $I\subset\T$.
A theorem of Hunt, Muckenhoupt and Wheeden \cite[Chapter VI, Theorem 6.2]{gar} 
asserts that $\mathfrak{w}$ satisfies condition $A_p$ if and only if 
\begin{equation}
\label{HMW}
\int_\T |\widetilde{\phi}|^p\,\mathfrak{w}\,d\Lambda
\leq C \int_\T |\phi|^p\,\mathfrak{w}\,d\Lambda,
\qquad \phi\in L^1(\T),
\end{equation}
where $C$ depends only on $\{\mathfrak{w}\}_{A_p}$.
In \eqref{HMW}, the 
assumption $\phi\in L^1(\T)$ is just a means to ensure that $\widetilde{\phi}$
is well defined.

\subsection{Pseudo-holomorphic Hardy spaces}
\label{Gpalpha}
Given $\alpha\in L^r(\D)$ for some $r\in[2,\infty)$ and $p\in(1,\infty)$,
we define the Hardy space $G^p_\alpha(\D)$ of those
$w\in L^\gamma_{loc}(\D)$ with $\gamma>r/(r-1)$ that satisfy \eqref{eq:w}, such that
\begin{equation} 
\label{esssuppw}
\Vert w\Vert_{G^p_\alpha(\D)}:= 
\sup_{0<\rho  < 1}
\Bigl(\int_{\T_\rho} \left\vert w(\xi)\right\vert^p |d\xi|\Bigr)^{1/p} <+\infty.
\end{equation}
Denote by $\mathcal{H}^p$ the Banach space of complex
measurable functions $f$ on
$\D$ such that $\text{ess.}\sup_{0<\rho<1}\,\rho\|f_\rho\|_{L^p(\T)}<+\infty$. 
Then $G^p_\alpha(\D)$ is identified with a real
subspace of $\mathcal{H}^p$. The fact that this subspace is closed  
(hence a Banach space in its own right) is a part of Theorem~\ref{trace2} below.
Note that if $w\in L^\gamma_{loc}(\D)$ satisfies \eqref{eq:w}, then 
$w\in W^{1,q}_{loc}(\D)$ for $q\in[1,2)$ by Lemma~\ref{expsf2}; hence the integral in \eqref{esssuppw} is indeed
finite for \emph{each} $\rho$ by the trace theorem.
Clearly $G^p_0(\D)=H^p$, but $G^p_\alpha(\D)$ is not
a complex vector space when $\alpha\not\equiv0$.
Spaces $G^1_\alpha(\D)$ and $G_\alpha^\infty(\D)$
could be defined similarly, but we shall not consider them.

For $r>2$, such classes of functions were apparently introduced in 
\cite{Musaev1} and subsequently considered in
\cite{Klimentov1,Klimentov2,klimentov2006,BLRR,
EfendievRuss, TheseYannick,BFL}.
In contrast to these studies,
our definition is modeled after \eqref{esssuppa}
rather than \eqref{esssupp}, that is, 
integral means in \eqref{esssuppw} are with respect to
arclength\footnote{Thus, it would be more 
appropriate to call $G^p_\alpha(\D)$ a pseudo-holomorphic Smirnov space.} 
and \emph{not} normalized arclength.
This is not important when $r>2$, but becomes essential\footnote{When 
$r=2$, $w$ may fail to satisfy condition \eqref{esssupp} 
even though it meets
\eqref{eq:w} and \eqref{esssuppw}. The problem lies with \emph{small}
 values of $r$,
as $w$ needs not be locally bounded on $\D$.} if $r=2$.

Below, we do consider the case $r=2$ and stress 
topological connections with holomorphic Hardy spaces which are new even
when $r>2$, see Theorem~\ref{trace2}~(iii).

By Lemma~\ref{expsf2}, each solution to \eqref{eq:w} 
in $L^\gamma_{loc}(\D)$, $\gamma>r/(r-1)$,
factors as $w=e^sF$ where 
\begin{equation}
\label{contnsa}
\|s\|_{W^{1,r}(\D)}\leq C(r)\|\alpha\|_{L^r(\D)}
\end{equation}
and $F$ is holomorphic in $\D$. Moreover, if $w\not\equiv0$,
one can impose $\text{Im}\,\tr_\T s=0$ and $\int_\T\text{Re}\,s=0$ or
$\text{Re}\,\tr_\T s=0$ and $\int_\T\text{Im}\,s=0$ to get unique factorization.
To distinguish between these two factorizations,
we write $w=e^{s^\mathfrak{r}}F^\mathfrak{r}$ in the
first case, and 
$w=e^{s^\mathfrak{i}}F^\mathfrak{i}$ in the second one; that is, 
$s^\mathfrak{r}$ is real  on $\T$ and  $s^\mathfrak{i}$ is 
pure imaginary there. If $w\equiv0$, we put $F^\mathfrak{r}=F^\mathfrak{i}=0$
and do not define $s^\mathfrak{r}$ and $s^\mathfrak{i}$.
When $w\not\equiv0$ (hence $w$ is a.e. nonzero), it follows
from the proof of Lemma~\ref{expsf2} that if we let 
\begin{equation}
\label{normreelR}
\mathcal{R}(\beta)(z):=
-\overline{\mathcal{C}(\beta)(1/\bar z)}=
\frac{1}{2\pi i}\int_{\D}
\frac{z\bar \beta(\xi)}{1-\bar \xi z }
\,d\xi\wedge\overline{d\xi},\,\,\, \beta\in L^r(\D),\ z\in\DD,
\end{equation}
then $s^\mathfrak{r}$ is given by
\begin{equation}
\label{normreel}
s^\mathfrak{r}=\mathcal{C}(\alpha\bar w/w)-
\mathcal{R}(\alpha\bar w/w)
\end{equation}
while $s^\mathfrak{i}$ is given by
\begin{equation}
\label{normim}
s^\mathfrak{i}=\mathcal{C}(\alpha\bar w/w)+
\mathcal{R}(\alpha\bar w/w).
\end{equation}

Indeed, it is easy to check that
$\mathcal{R}(\alpha\bar w/w)$ is a holomorphic function
in $W^{1,r}(\D)$ having zero mean on $\T$ and assuming conjugate values to
$-\mathcal{C}(\alpha\bar w/w)$ there.

From \eqref{contnsa} which is valid both for
$s^{\mathfrak{r}}$ and $s^{\mathfrak{i}}$ we get that if $r>2$ then
\begin{equation}
\label{normes}
\|e^{\pm s^\mathfrak{r}}\|_{W^{1,r}(\D)}\leq C(r,\|\alpha\|_{L^r(\D)})
\,\,\, \text{and}\,\,\,
\|e^{\pm s^\mathfrak{i}}\|_{W^{1,r}(\D)}\leq C(r,\|\alpha\|_{L^r(\D)}).
\end{equation}
For $r=2$ and for $1<q<2$, we only deduce from
\eqref{contnsa}
and Proposition~\ref{embexpS} that 
\begin{equation}
\label{normes2}
\|e^{\pm s^\mathfrak{r}}\|_{W^{1,q}(\D)}\le C(q,\|\alpha\|_{L^2(\Omega)})
\,\, \text{and}\,\,
\|e^{\pm s^\mathfrak{i}}\|_{W^{1,q}(\D)}\le  C(q,\|\alpha\|_{L^2(\Omega)}).
\end{equation}

\begin{itemize}
\item When $r>2$, we conclude from \eqref{normes} and the Sobolev embedding 
theorem that $e^{\pm s^\mathfrak{r}}$ and $e^{\pm s^\mathfrak{i}}$
are continuous and bounded independently of $w$ on $\overline{\D}$.
Hence, $w$ belongs to $G^p_\alpha(\D)$ if and only if 
$F^\mathfrak{r}$ or $F^\mathfrak{i}$ lies in $H^p$ (in which case  
both do). This way $G^p_\alpha(\D)$ inherits many properties of $H^p$.
In particular, each $w\in G_\alpha^p(\D)$ 
has a nontangential limit a.e.
on $\T$, denoted again by $w$ or $w_{\T}$ for emphasis, which is also 
the limit of $w_\rho$ as $\rho\to1^-$ in
$L^p(\T)$. Moreover, $\|w_{\T}\|_{L^p(\T)}$ 
is a norm equivalent
to \eqref{esssuppa} on $G^p_\alpha(\D)$, and we might as well have used 
\eqref{esssupp} to define the latter. Also, from Theorem~\ref{BNpar},
we infer that condition \eqref{Blaschke} characterizes the zeros of
non identically vanishing functions in $G^p_\alpha(\D)$\footnote{When $r=2$,
this property has no simple analog since $w$
is only defined  $B_{1,2}$-quasi-everywhere.}.
\item If $r=2$, all we conclude {\it a priori} 
from \eqref{normes2}, Lemma~\ref{bornewr}, and 
H\"older's inequality  is
that $F^\mathfrak{r}$ and $F^\mathfrak{i}$ belong to
$\cap_{1\leq\ell<p}H^\ell$ if $w\in G^p_\alpha(\D)$. In the other direction,
$w\in\cap_{1\leq\ell<p}G_\alpha^p(\D)$ if $F^\mathfrak{r}$ or 
$F^\mathfrak{i}$ lies in $H^p$. To clarify the matter,
one should realize that factorizations 
$w=e^{s^\mathfrak{r}}F^\mathfrak{r}$ and 
$w=e^{s^\mathfrak{i}}F^\mathfrak{i}$ no longer play equivalent roles.
For it may happen that $w\in G^p_\alpha(\D)$
and $F^\mathfrak{r}\notin H^p$. In fact, if we let
\begin{equation}
\label{exw}
w(z):=\frac{1}{\log(3/|z-1|)\,(z-1)^{1/p}},\qquad z\in\D,
\end{equation}
we get  that
$$
\Bigl|\frac{\bar \partial w(z)}{\overline{w(z)}}\Bigr|=
\bigl(2|z-1|\log(3/|z-1|)\bigr)^{-1};
$$
hence, $w\in G^p_\alpha(\D)$ with
$\alpha:=\bar\partial w/\bar w\in L^2(\D)$,
but the factorization
$$
w(z)=e^{\log\log(3/|z-1|)-a }\,
\,e^a(z-1)^{-1/p},\qquad a:=\int_\T  \log\log(3/|z-1|)\,d\Lambda(z),
$$
is such that $F^\mathfrak{r}=e^a(z-1)^{-1/p}\notin H^p$.

On the other hand, $w\in G^p_\alpha(\D)$ if and only if
$F^\mathfrak{i}\in H^p$. Assume indeed that $0\not\equiv w\in G^p_\alpha$.
Since $F^\mathfrak{i}\in H^\ell$ for $1\leq\ell<p$
and $e^{s^\mathfrak{i}_\rho}$ converges to
$e^{s^\mathfrak{i}}$ in $W^{1,q}(\D)$ for all $q\in[1,2)$ by Proposition~\ref{embexpS},  it follows from Lemma~\ref{bornewr} 
and H\"older's inequality  that
$\tr_{\T}  w_\rho$ converges as $\rho\to1^-$ to 
$e^{\tr_\T s^\mathfrak{i}}F^\mathfrak{i}_{|\T}$ in 
$L^\lambda(\T)$, for every $\lambda\in[1,p)$. 
Moreover, as $\tr_{\T}w_\rho$ remains bounded in $L^p(\T)$
by \eqref{esssuppw}, it 
converges weakly there to $e^{\tr_\T s^\mathfrak{i}}F^\mathfrak{i}_{|\T}$ 
when  $\rho\to1^-$,
since this  is the only weak limit possible granted 
the convergence of  $\tr_{\T}  w_\rho$ in $L^\lambda(\T)$.
In particular,  $e^{\tr_\T s^\mathfrak{i}}F^\mathfrak{i}_{|\T}\in L^p(\T)$,
and since $|e^{\tr_\T s^\mathfrak{i}}|\equiv 1$
we conclude that $F^\mathfrak{i}_{|\T}\in L^p(\T)$, and hence $F^\mathfrak{i}
\in H^p$. Conversely, if  $F^\mathfrak{i}\in H^p$,
then $w$ satisfies \eqref{esssuppw}
by Corollary~\ref{condHm}.
\end{itemize}

The fact that $\tr_{\T}  w_\rho$ converges  strongly in $L^p(\T)$ as $\rho\to1^-$, and not just weakly 
as we showed above, is a part of the next theorem, whose assertion $(iii)$ is new even for $r>2$.

\begin{thm}
\label{trace2}
Let $\alpha\in L^r(\D)$ with $2\leq r<\infty$ and 
fix $p\in(1,\infty)$. 
\begin{itemize}
\item[(i)] Each $w\in G^p_\alpha(\D)$ has a trace
$w_\T$ on $\T$ given by
\begin{equation}
\label{deftr2}
w_\T:=\lim_{\rho\to1^-}\tr_\T w_\rho\quad\text{in $L^p(\T)$.}
\end{equation}
When $r>2$, the function $w_\T$ is also the 
non-tangential limit of $w$ a.e. on $\T$.
\item[(ii)] For some $C>0$ depending only on
$|\alpha|$  and $p$ we have 
\begin{equation}
\label{normeq}
\|w_\T\|_{L^p(\T)}\leq \|w\|_{G^p_\alpha(\D)}\leq C\|w_\T\|_{L^p(\T)},
\end{equation}
and $G^p_\alpha(\D)$ is a real Banach space on which
$\|w_\T\|_{L^p(\T)}$ is a norm equivalent to \eqref{esssuppw}. 
\item[(iii)] The map $w\mapsto F^\mathfrak{i}$ is a homeomorphism from
$G^p_\alpha(\D)$ onto $H^p$. When $r>2$, the map
$w\mapsto F^\mathfrak{r}$ is also  such a homeomorphism.
\item[(iv)]  If $w\in G^p_\alpha(\D)$ and 
$w_\T\in L^q(\T)$ for some $q\in(1,\infty)$, then $w\in G^q_\alpha(\D)$.
A non-negative function $h\in L^p(\T)$ is such that $h=|w_\T|$ for some
nonzero $w\in G^p_\alpha(\D)$ if and only if $\log h\in L^1(\T)$.
\end{itemize}
\end{thm}

\begin{proof}
If $r>2$, all the properties except $(iii)$ follow from their $H^p$-analogs
{\it via} the continuity and uniform boundedness
of  $e^{\pm s^\mathfrak{r}}$ or $e^{\pm s^\mathfrak{i}}$
discussed earlier in this section, see also
\cite{Musaev1,Klimentov1,BLRR,BFL}.

We postpone the proof of $(iii)$ and assume for now that $r=2$.
Take $w\in G_\alpha^p(\D)\setminus\{0\}$ and put $s=s^\mathfrak{i}$,
$F=F^\mathfrak{i}$ to simplify notation. 
To prove \eqref{deftr2} we need to verify that given a sequence 
$(\rho_n)\subset(0,1)$ tending to $1$, one can extract a subsequence 
$(\rho_{n_k})$ such that $\tr_\T\, w_{\rho_{n_k}}$ converges to 
$e^{\tr_\T \,s}F_{|\T}$ in $L^p(\T)$. 

Since $s_\rho$ converges to
$s$ in $W^{1,2}(\D)$, we get from Lemma~\ref{bornewr}
that $\tr_\T s_\rho$ 
converges to $\tr_\T s$ in $L^\ell(T)$, as $\rho\to1^-$, for all
$\ell\in[1,\infty)$. Moreover, as we pointed out before the theorem, 
$F\in H^p$, and hence $(F_\rho)_{|\T}$ 
converges to $F_{|\T}$ in $L^p(\T)$.
Extracting  if necessary a subsequence from $(\rho_n)$ 
(still denoted by $(\rho_n)$), we can assume that
$\tr_\T s_{\rho_n}$ (resp.  
$\left(F_{\rho_n}\right)_{|\T}$) also 
converges pointwise a.e. on $\T$ to $\tr_\T s$ (resp. $F_{|\T}$).
Now, Corollary~\ref{condHm}, applied with 
$ps$ instead of  $s$, 
implies that 
$\|e^{ps_\rho} F_\rho\|_{L^p(\T)}$ 
is uniformly bounded as $\rho\to1^-$. Therefore, as the weak limit 
coincides with 
the pointwise limit when both exist by Egoroff's theorem, 
there is a subsequence $(\rho_{n_k})$ such that 
$\Bigl(e^{p{\rm Re}\, s_{\rho_{n_k}}}\bigl|F_{\rho_{n_k}}\bigr|\Bigr)$ 
converges weakly to $|F|$ in $L^p(\T)$. 
Letting $1/p+1/p'=1$, this means that
for each test function $\Theta\in L^{p'}(\T)$ we have:
\begin{equation}
\label{convfp}
\lim_{k\to\infty}\Bigl|\int_\T e^{p {\rm Re}\, s_{\rho_{n_k}}(z)}\bigl|F_{\rho_{n_k}}(z) \bigr|\Theta(z)\, |dz|-\int_\T |F(z)|\Theta(z)\,|dz|\Bigr|= 0.
\end{equation}
Set $\Theta_k=\bigl|F_{\rho_{n_k}}\bigr|^{p-1}\in L^{p'}(\T)$. Convergence of
$(F_\rho)_{|\T}$ to $F_{|\T}$ in $L^p(\T)$ implies easily that
$\Theta_k$ converges to $|F|^{p-1}$ in $L^{p'}(\T)$. In view of
\eqref{convfp}, this yields
$$
\lim_{k\to\infty}\Bigl|\int_\T e^{p {\rm Re}\, s_{\rho_{n_k}}(z)}\bigl|F_{\rho_{n_k}}(z)\bigr|^{p}\, |dz|-\int_\T |F(z)|^p\,|dz|\Bigr|= 0.
$$
Therefore,
$\bigl\| \tr_\T w_{\rho_{n_k}}\bigr\|_{L^p(\T)}=
\bigl\| e^{{\rm Re}\, s_{\rho_{n_k}}}F_{\rho_{n_k}} \bigr\|_{L^p(\T)}$ 
tends  to $\|e^{\tr_\T s}F_{|\T}\|_{L^p(\T)}=
\bigl\| F_{|\T} \bigr\|_{L^p(\T)}$ when $k\to\infty$. 
However, from the discussion 
before the theorem, we know  that 
$\tr_{\T}w_{\rho_{n_k}}$ converges weakly to $e^{\tr_\T s}F_{|\T}$ in 
$L^p(\T)$,
so by uniform convexity of $L^p(\T)$ the convergence must in fact be strong
because, as we just showed, the 
norm of the weak limit is the limit of the norms \cite[Theorem 3.32]{brezis}.
This proves $(i)$.

Next, we observe by the absolute continuity of $|\alpha|^2dm$ that for every  
$\varepsilon>0$ there is $\omega(\varepsilon)>0$ for which
$\|\alpha\|_{L^2(Q_{\omega(\varepsilon)}\cap\D)}<\varepsilon$ as soon as 
$Q_{\omega(\varepsilon)}$ is a cube of sidelength $\omega(\varepsilon)$. 
Thus, in view of \eqref{normim}, 
we can apply Proposition~\ref{evalderCauchy} to 
$\beta:=\alpha \bar w/w$
and obtain a strictly positive function  $\widetilde{\omega}$ on $\RR^+$,
\emph{depending only on} $|\alpha|$, such that
\begin{equation}
\label{borderCa}
\|\partial s\|_{L^2(Q_{\widetilde{\omega}(\eta)}\cap\D)}+
\|\bar \partial s \|_{L^2(Q_{\widetilde{\omega}(\eta)}\cap\D)}
<\eta
\end{equation}
as soon as 
$Q_{\widetilde{\omega}(\eta)}$ is a cube of sidelength 
$\widetilde{\omega}(\eta)$. A fortiori, \eqref{borderCa} holds with
$\text{Re}\,s$ instead of $s$. Now, picking any $\gamma\in (0,\pi/2)$ and 
recalling that $|w_\T|=|F_{|\T}|$ because 
${\rm Re}\,s\in W^{1,2}_{0,\R}(\D)$,
we deduce 
from \eqref{ntborne} and Theorem~\ref{somprodBN} applied to $f={\rm Re}\,s$
and $g=e^{i{\rm Im}\,s}F$ that
the right inequality in \eqref{normeq} holds. In another connection,
the left inequality is obvious from \eqref{deftr2}. 
To show that $G^p_\alpha(\D)$
is a Banach space, consider a sequence $(w_n)\subset G_\alpha^p(\D)$ 
converging  in  $\mathcal{H}^p$ to some function $w$. We must prove that
$w\in G^p_\alpha$. We can assume  $w\not\equiv0$,  therefore
$w_n\not\equiv0$ for $n$ large enough.
Convergence in 
$\mathcal{H}^p$ being  stronger than in $L^p(\D)$, 
{\it a fortiori} $w_n$ converges to $w$ in $\mathcal{D}'(\D)$
and, moreover, some subsequence, again denoted by $w_n$, 
converges pointwise a.e. to $w$. Besides, if we write 
$w_n=e^{s_n} F_n$ where we mean as before that $s_n=s_n^\mathfrak{i}$ 
and $F_n=F_n^\mathfrak{i}$,  we get from \eqref{normes2}
that the sequence $(e^{s_n})$ is bounded in 
$W^{1,q}(\D)$ for each $q\in(1,2)$. 
Therefore, by the Sobolev embedding theorem,
$(e^{s_n})$ is bounded in $L^\ell(\D)$ for each
$\ell\in [1,\infty)$. In addition, since $|e^{\tr_\T s_n}|\equiv1$,
it follows from \eqref{normeq} that $(F_n)$ is bounded in $H^p$,  
hence also in $L^\ell(\D)$ for each 
$\ell\in(1,2p)$ by \eqref{sum2p}.  Altogether, by H\"older's inequality,
$(w_n)$ is bounded in $L^\gamma(\D)$ for some $\gamma>2$. Consequently, 
some subsequence converges weakly in $L^\gamma(\D)$, and since the weak
limit coincides with the pointwise limit, if it exists, we conclude 
that the weak limit is $w$. In particular, $w\in L^\gamma(\D)$.
Moreover, by H\"older's inequality,
$(\alpha \bar w_n)$ is bounded in $L^t(\D)$ for some
$t>1$, and arguing as before we get that some 
subsequence (again denoted by $(\alpha \bar w_n)$) converges weakly 
to $\alpha \bar w$ there.
Thus, passing to the distributional limit in the relation
$\bar\d w_n=\alpha\bar w_n$, we obtain \eqref{eq:w} so that 
$w\in G^p_\alpha(\D)$. This proves $(ii)$.

We already know from Theorem~\ref{BNpar} 
and  the discussion before Theorem~\ref{trace2}
that the map $w\mapsto F^\mathfrak{i}$ is  
bijective from $G^p_\alpha(\D)$ to $H^p$. 
Since $|w_\T|=|{F^\mathfrak{i}}_{|\T}|$,
it is clear from \eqref{normeq}  that this map and its inverse are
continuous at $0$.
Let now $w_n$ converge to $w\not\equiv0$ in $G^p_\alpha(\D)$ and write
$w_n=e^{s_n^\mathfrak{i}}F_n^\mathfrak{i}$,  
$w=e^{s^\mathfrak{i}}F^\mathfrak{i}$. We claim that some subsequence
of $F_n^\mathfrak{i}$ converges to $F^\mathfrak{i}$ in $H^p$ and this 
establishes continuity of the map at every point. 
As  ${F_n^\mathfrak{i}}_{|\T}$ is bounded in $L^p(\T)$ by \eqref{normeq},
some subsequence converges weakly there to $\Phi_{|\T}$ for some  
$\Phi\in H^p$. Thus, replacing $w_n$ by a subsequence 
(again denoted by $w_n$),
we may assume by the Cauchy formula that $F_n^\mathfrak{i}$ converges 
locally uniformly to $\Phi$ on $\D$. Note that $\Phi\not\equiv0$ for 
otherwise, in view 
of \eqref{normes2}, we would have that $w_n$ converges to the zero 
distribution,
contradicting  that $w\not\equiv0$. In particular,    
$\alpha\bar F_n^\mathfrak{i}/F_n^\mathfrak{i}$ converges in
$L^2(\D)$ to  $\alpha \bar\Phi/\Phi$ by the dominated convergence theorem.
Since $\bar\d s_n^\mathfrak{i}=\alpha\bar F_n^\mathfrak{i}/F_n^\mathfrak{i}
\exp(-2i\text{Im}\,s_n^\mathfrak{i})$ by Lemma~\ref{arglisse} 
and 
$\|s_n^\mathfrak{i}\|_{W^{1,2}(\D)}$ is uniformly bounded by 
\eqref{contnsa},
we can argue as we did after \eqref{limsn} (put 
$\alpha_n\equiv\alpha\bar F_n^\mathfrak{i}/F_n^\mathfrak{i}$ and 
$\theta_0=0$ in the discussion there) to the 
effect that a subsequence, again denoted by $s_n^\mathfrak{i}$,
converges to some $\sigma\in W^{1,2}(\D)$ such that
$\text{Re}\,\tr_\T \sigma=0$ and $\int_\T\sigma=0$, both a.e. 
and in $W^{1,2}(\D)$. Refining the sequence if necessary, we can further
assume that $w_n$ converges a.e. to $w$. 
Taking pointwise limits we get  $w=e^\sigma\Phi$,
hence $\sigma=s^\mathfrak{i}$ and $\Phi=F^\mathfrak{i}$ by the uniqueness part
of Corollary~\ref{BNreal}. Thus, $F^\mathfrak{i}_{|\T}$ is the weak limit
of  ${F_n^\mathfrak{i}}_{|\T}$, and since  
$\|{F^\mathfrak{i}}\|_{L^p(\T)}=\|w\|_{L^p(\T)}$ is the limit of
$\|{F_n^\mathfrak{i}}\|_{L^p(\T)}=\|w_n\|_{L^p(\T)}$, the convergence 
in fact takes place in $L^p(\T)$, thereby proving the claim. 
Conversely, let $w_n=e^{s_n^\mathfrak{i}}F_n^\mathfrak{i}$ be a sequence in
$G_\alpha^p(\D)$ such that $F_n^\mathfrak{i}$ converges to $\Phi\not\equiv0$ 
in $H^p$. By Corollary~\ref{BNreal}, $\|s_n^\mathfrak{i}\|_{W^{1,2}(\D)}$
is bounded uniformly in $n$, and, as before,  
a subsequence, again denoted by $s_n^\mathfrak{i}$,
converges in $W^{1,2}(\D)$ to some $\sigma$ such that
$\text{Re}\,\tr_\T \sigma=0$ and $\int_\T\sigma=0$. 
Refining the sequence if necessary, we can 
assume in view of the trace theorem 
that $\tr_\T s_n^\mathfrak{i}$ converges pointwise a.e. on $\T$
to  $\tr_\T \sigma$. By the dominated convergence, $(w_n)_\T$ tends to
$e^{\tr_\T\sigma}\Phi_{|\T}$ in $L^p(\T)$. Using  \eqref{normeq} we obtain  
that $w_n$ converges in $G^p_\alpha(\D)$ to some 
$w=e^{s^\mathfrak{i}} F^\mathfrak{i}$, and by the
continuity proven before we conclude that
$\Phi=F^\mathfrak{i}$. This proves $(iii)$ when $r=2$.
That both $w\mapsto F^\mathfrak{i}$  and  $w\mapsto F^\mathfrak{r}$ are
homeomorphisms when $r>2$ is similar but easier because then $s\to e^s$ is
bounded and continuous from $W^{1,r}(\D)$ into 
$W^{1,r}(\D)\subset L^\infty(\D)$.

Finally, $(iv)$ follows from the corresponding properties of $H^p$ functions,
the fact that $w\in G^p_\alpha$ if and 
only it $F^\mathfrak{i}\in H^p$, and the equality 
$|w_\T|=|F^\mathfrak{i}_{|\T}|$.
\end{proof}

\begin{rem} When $r=2$, $w_\T$ in  Theorem~\ref{trace2} 
is not necessarily the non-tangential limit of $w$.
Indeed, if $(z_n)\subset\D$ is nontangentially dense on $\T$,
then $s(z):=\sum_n2^{-n}\log\log 2/|z-z_n|$ lies in $W^{1,2}(\D)$
so that $e^s\in G^p_\alpha(\D)$ for all $p\in(1,\infty)$ 
with $\alpha:=\bar\partial s$ by
Lemma~\ref{bornewr}. Yet, $e^s$ is not even 
nontangentially bounded at a single $\xi\in\T$.
\end{rem}

\section{The generalized conjugation operator}
\label{x20}

The M. Riesz theorem may be rephrased as follows.
Given $\psi\in L^p_\R(\T)$ with $p\in(1,\infty)$, the problem of finding a 
holomorphic function $f$ in $\D$ such that $\text{Re}\,\tr_\T f_\rho$ tends to
$\psi$ in $L^p(\T)$ has a solution in $H^p$ which is unique up to an 
additive imaginary constant. In fact, if we normalize it to have
mean $\int_\T\psi/2\pi+ic$ on $\T$, then $f_{|\T}=\psi+i\tilde{\psi}+ic$ and 
we have 
$\|f\|_{H^p}\le C(\|\psi\|_{L^p(\T)}+|c|)$ for some $C$ 
depending only on $p$.

The corresponding problem for pseudo-holomorphic functions, {\it i.e.}
for solutions to \eqref{eq:w} when $\alpha\not\equiv0$, turns out to have
a similar answer in $G^p_\alpha$ as long as $\alpha\in L^r(\D)$ for some
$r\ge2$. 
When $r>2$ this was essentially proven in \cite{Klimentov1},
see also \cite{BLRR} and
\cite{BFL}. 
More precisely:

\begin{thm}[\cite{Klimentov1},\cite{BLRR},\cite{BFL}]
\label{Dirr}
Let $\alpha\in L^r(\D)$ with $2< r\leq\infty$ and 
$1<p<\infty$. For every  
$\psi\in L^p_\R(\T)$ and $c\in\R$ there is a unique 
$w\in G^p_\alpha(\D)$ such that 
$\mbox{Re}\,w_\T=\psi$ and $\int_\T \mbox{Im}\,w_\T=c$. 
Moreover, $\|w\|_{G^p_\alpha(\D)}\le C(\|\psi\|_{L^p(\T)}+|c|)$,
where $C$ depends only on $p$ and $r$.
\end{thm}

Theorem~\ref{Dirr} generalizes  the M. Riesz theorem:
for every $\psi\in L^p_\R(\T)$ and $c\in\R$ 
there is a unique $\psi^\sharp_c\in L^p_\R(\T)$ 
(a generalized conjugate of $\psi$) such that
$\int_\T\psi^\sharp_c=c$ and $\psi+i\psi^\sharp_c=w_\T$ for some 
$w\in G^p_\alpha(\D)$. Moreover, 
$\|\psi^\sharp\|_{L^p(\T)}\le C(\|\psi\|_{L^p(\T)}+|c|)$.
The theorem below extends this result to the case $r=2$ where
solutions to \eqref{eq:w} may be locally unbounded.

\begin{thm}
\label{Dirichlet2}
Let $\alpha\in L^2(\D)$ and $1<p<\infty$. 
For every  
$\psi\in L^p_\R(\T)$ and $c\in\R$ there is a unique
$w\in G^p_\alpha(\D)$ such that 
$\text{Re}\,w_\T=\psi$ and $\int_\T \text{Im}\,w_\T=c$. 
Moreover, 
\begin{equation}
\label{estu2}
\|w\|_{G^p_\alpha(\D)}\le C(\|\psi\|_{L^p(\T)}+|c|),
\end{equation}
where $C$ depends only on $p$ and $|\alpha|$.
\end{thm}

\begin{proof}
We first show existence. Assume that 
$\psi$ and $c$ are not both zero; otherwise $w\equiv0$ will do.

Let $(\alpha_n)$ be a sequence of functions in $L^\infty(\D)$ converging
to $\alpha$ in $L^2(\D)$. 
By Theorem~\ref{Dirr}, for every $n$ there exists
$w_n\in G_{\alpha_n}^p(\D)$ such that $\mbox{Re}\,{w_n}_{|\T}=\psi$ and
$\int_\T \mbox{Im}\,w_n=c$. 
Notations being as in Section~\ref{Gpalpha},
let us write $w_n=e^{s_n^\mathfrak{r}}F_n^\mathfrak{r}$ where 
$s_n^\mathfrak{r}\in W^{1,2}(\D)$ is real with zero mean on $\T$ while
$F_n^\mathfrak{r}\in H^p$.
Below, we drop the superscript
$\mathfrak{r}$ for simplicity.

It follows from \eqref{contnsa} that
$\|s_n\|_{W^{1,2}(\D)}\leq C_0\|\alpha_n\|_{L^2(\D)}$ for some 
absolute constant $C_0$, hence 
$\|s_n\|_{W^{1,2}(\D)}$ is bounded uniformly in $n$.
In view of the
Rellich--Kondrachov theorem, we 
can find a subsequence, again denoted by $(s_n)$,  
converging to some function $s$ both pointwise on $\D$ and  in
$L^\ell(\D)$ for  all $\ell\in[1,\infty)$. By the trace theorem and 
the non integral version of the Rellich--Kondrachov theorem, we 
may further assume that  $\tr_\T s_n$ 
converges to some function $h$ both pointwise a.e. on $\T$ and in 
$L_\R^\ell(\T)$. Moreover, 
convergence of $\alpha_n$ to $\alpha$
in $L^2(\DD)$ entails, because of
\eqref{normes2}, 
that  $e^{\pm s_n}$ are bounded in
$W^{1,q}(\D)$, independently of $n$ and $\psi$, for each $q\in[1,2)$.
So, invoking again the trace and the Rellich--Kondrachov theorems, 
we may assume upon refining $s_n$ further that
$e^{\pm \tr_\T s_n}$ converges to their pointwise limits $e^{\pm h}$ in  
$L^\ell(\T)$, for all $\ell\in[1,\infty)$. 

Thus, by H\"older's inequality, 
$\mbox{Re}\,(F_n)_{|\T}=e^{-\tr_\T s_n}\psi$  
converges to $e^{-h}\psi$ in $L^\lambda(\T)$
for any $\lambda\in[1,p)$. 
Continuity of the conjugate
operator now implies that $\widetilde{\mbox{Re}\,(F_n)_{|\T}}$
in turn converges  to $\widetilde{e^{-h}\psi}$ in $L^\lambda(\T)$.
Since $\int_\T \mbox{Im}\,w_n=c$, we see by inspection that 
$\mbox{Im}\,(F_n)_{|\T}=\widetilde{\mbox{Re}\,(F_n)_{|\T}}+c_n $
where the constant $c_n$ is such that
\begin{equation}
\label{cim}
c_n\int_\T e^{\tr_\T s_n}+\int_\T e^{\tr_\T s_n}
\widetilde{\mbox{Re}\,(F_n)_{|\T}}=c.
\end{equation}
The first integral in \eqref{cim} converges to 
$\int_\T e^h>0$, and the second integral  converges to
$\int_\T e^{h} \,\widetilde{e^{-h}\psi}$ by H\"older's inequality.
Therefore, $(c_n)$ converges to 
\begin{equation}
\label{c0}
c_0:=\left(c-\int_\T e^{h} \,\widetilde{e^{-h}\psi}\right)/\int_\T e^h, 
\end{equation}
and subsequently 
$(F_n)_{|\T}$ converges to 
\begin{equation}
\label{trF}
F_{\T}:=e^{-h}\psi +i\,\widetilde{e^{-h}\psi}+ic_0
\end{equation} in 
$L^\lambda(\T)$, for all $ \lambda\in[1,p)$. 
Thus, $F_n$ converges in $H^\lambda$ to $F$, the Poisson integral of
$F_\T$. Note that $F$ is not identically zero; otherwise, $\psi\equiv0$ and $c=0$, contrary to
our initial assumption.

The above argument and the dominated convergence theorem  
give us that $\alpha_n e^{-2i\text{\rm Im}\,s_n}\bar F_n/F_n$ converges to
$\alpha e^{-2i\text{\rm Im}\,s}\bar F/F$ in $L^2(\D)$.
Next, $\bar\d s_n= \alpha_n \times\allowbreak\times e^{-2i\text{\rm Im}\,s_n}\bar F_n/F_n$. 
By Lemma~\ref{arglisse},
applying \eqref{intdbar} with $A=s_n-s_m$, 
$a=\bar\partial s_n-\bar\partial s_m$, 
$\theta_0=-\pi/2$, $\psi\equiv0$, and $\lambda=0$, 
we conclude that $(s_n)$ is a Cauchy sequence in $W^{1,2}(\D)$  
which must therefore converge to $s$. 
Hence, $s\in W^{1,2}(\D)$ and $h=\tr_\T s$.
Since we get in the limit
that $\bar\partial s=(\alpha \bar F/F) e^{-2i\text{Im}\,s}$, we see from
Lemma~\ref{arglisse} that $w:=e^sF$ satisfies \eqref{eq:w}.
Moreover, if we write $w=e^{s^\mathfrak{r}}F^\mathfrak{r}$ 
in the notation of Section~\ref{Gpalpha}, 
we find that $s^\mathfrak{r}=s$ and 
$F^\mathfrak{r}=F$ because $s$ inherits from $s_n$ the properties
$\text{Im}\, \tr_{\T}s=0$ and  $\int_\T \text{Re}\,s=0$.
As $F\in H^\lambda$ for all $ \lambda\in[1,p)$, we further deduce
from the discussion before \eqref{exw} that $w\in G^\lambda_\alpha$
for all such $\lambda$. By inspection of \eqref{trF} we get
\begin{multline}
\label{trw}
w_\T=e^{\tr_\T s}F_{|\T}=e^{\tr_\T s}(e^{-\tr_\T s}\psi +i
\,\widetilde{e^{-\tr_\T s}\psi}+ic_0)\\=
\psi+i\bigl(e^{\tr_\T s}\,\widetilde{e^{-\tr_\T s}\psi}+e^{\tr_\T s}c_0\bigr),
\end{multline}
where we use the fact that $h=\tr_\T s$ is real-valued.
In particular, \eqref{trw} entails that $\mbox{Re}\,w_\T=\psi$.

To show that 
$w\in G^p_\alpha(\D)$, we must prove in view of Theorem~\ref{trace2}
that $w_\T\in L^p(\T)$. To do this, note that 
$\psi\in L^p(\T)$ by assumption and that $e^{\tr_\T s}c_0\in L^p(\T)$ 
by the trace and Sobolev embedding theorems.
Furthermore,
$p\,{\tr_\T s}\in W^{1/2,2}(\T)\subset VMO(\T)$ by \eqref{VMOW1/2}.  
By Lemma~\ref{ApVMO}, $e^{p{\tr_\T s}}$ satisfies condition $A_p$. Thus, using \eqref{HMW}, we obtain 
\begin{equation}
\label{uM}
\|e^{\tr_\T s}\,\widetilde{e^{-\tr_\T s}\psi}\|_p^p
\leq C''\,\|\psi\|_p^p,\qquad C''=C''(\{e^{p{\tr_\T s}}\}_{A_p});
\end{equation}
in view of \eqref{trw} we have  
$w_\T\in L^p(\T)$.
This gives the existence part of Theorem~\ref{Dirichlet2}.

As for uniqueness, let $w_1, w_2\in G^p_\alpha(\D)$ be two solutions.
Set $v:=w_1-w_2\in G^p_\alpha(\D)$, so that 
$\mbox{Re}\,v_\T=0$, 
$\int_\T \mbox{Im}\,v_\T=0$.  If we write 
$v=e^{\sigma^\mathfrak{r}}\Phi^\mathfrak{r}$,
we observe that $\text{Re}\,(\Phi^\mathfrak{r})_{|\T}\equiv0$, and hence 
the $H^\lambda$ function $\Phi^\mathfrak{r}$, $1\leq\lambda<p$,
is a pure imaginary constant, say $\zeta$. 
Thus, $v=\zeta e^s$ 
and the relations $\int_\T \mbox{Im}\,v_\T=0$, 
$\int_\T e^{\tr_\T \sigma}>0$ 
give us 
$\zeta=0$ so that $v=0$, as desired.

Finally, we verify \eqref{estu2}. By \eqref{normeq}, it suffices to prove that
$$
\|w_\T\|_{L^p(\T)}\le C(\|\psi\|_{L^p(\T)}+|c|),
$$
where $C$ depends only on $p$.
By  \eqref{uM}, \eqref{trw}, \eqref{c0}  and H\"older's inequality, 
we need only
establish that $\|e^{\tr_\T s}\|_{L^p(\T)}$, $\{e^{p{\tr_\T s}}\}_{A_p}$,
and $1/\int_\T e^{\tr_\T s}$
are bounded from above independently of $\psi$.
We pointed out earlier in the proof that
$e^{s_n}$ are bounded in
$W^{1,q}(\D)$, independently of $n$ and $\psi$, for each $q\in[1,2)$.
Since $s_n$ tends to $s$ in $W^{1,2}(\D)$, boundedness of 
$\|e^{\tr_\T s}\|_{L^p(\T)}$ 
follows from Proposition~\ref{embexpS} and the (non-integral version of) the 
Sobolev embedding theorem.
Next, \eqref{contnsa} yields that
$\|s\|_{W^{1,2}(\D)}\leq C_0\|\alpha\|_{L^2(\D)}$ for some 
absolute constant $C_0$.
 Thus,
using concavity of $\log$, the Schwarz inequality, and the trace theorem,
we get for some absolute constant $C_1$ that
\begin{multline*}
\log\left( \frac{1}{2\pi}\int_\T e^{s(\xi)}\,|d\xi|\right)\geq
\frac{1}{2\pi}\int_\T s(\xi)\,|d\xi|\\ \ge
-\|s\|_{L^2(\T)}\ge-C_1\|s\|_{W^{1,2}(\D)}\geq
-C_0C_1\|\alpha\|_{L^{2}(\D)},
\end{multline*}
showing that  $\int_\T e^{\tr_\T s}\ge \exp\{-C_0C_1\|\alpha\|_{L^2(\D)}\}$.

Finally, to majorize 
$\{e^{p{\tr_\T s}}\}_{A_p}$
independently of $\psi$, it suffices by Lemma~\ref{ApVMO} to prove that 
$M_{\tr_\T s}(J)$ (see definition \eqref{defMhG}) 
can be made arbitrarily small as $\Lambda(J)\to 0$, 
uniformly with respect to $\psi$, as $J$ ranges over open arcs on $\T$.
Let $\omega$ be 
be a strictly
positive function on $(0,+\infty)$ such 
that
$\|\alpha\|_{L^2(Q_{\omega(\varepsilon)}\cap\D)}<\varepsilon$ as soon as 
$Q_{\omega(\varepsilon)}$ is a square  of sidelength $\omega(\varepsilon)$. 
By \eqref{normreel} and 
Proposition~\ref{evalderCauchy}, there is
a strictly positive function  $\widetilde{\omega}$ on  $(0,+\infty)$, 
depending only on $\omega$, such that \eqref{borderCa} holds.
Now, if $\Lambda(J)<1$, it is elementary to check 
that $R(J,\Lambda(J))$ ({\it cf.} definition \eqref{defR})
 is contained in a square of
sidelength $\Lambda(J)$. Therefore, if we pick
$\Lambda(J)<\min\{1/2,\widetilde{\omega}(\eta)\}$, we deduce from 
\eqref{VMOW1/2}  and Lemma~\ref{ecmoyW120} that
$M_{\tr_\T s}(J)\leq C_1 \eta$, where $C_1$ is an absolute constant.
This completes the proof of Theorem~\ref{Dirichlet2}.
\end{proof}

\section{Dirichlet problem for $\exp-W^{1,2}$ conductivity}
\label{sectDirichlet}

The following connection 
between pseudo-holomorphic functions and conductivity equations
is instrumental in \cite{ap} and was investigated in the context of 
pseudo-holomorphic Hardy spaces  in \cite{BLRR,BFL} when $r>2$.
We start by a 2-d isotropic conductivity equation with $\exp$-Sobolev smooth
coefficient:
\begin{equation}
\label{eq:u}
{\rm div}\left(\sigma\nabla u\right)=0\quad{\rm in }\ \Omega,
\qquad \sigma\geq0, \quad \log\sigma\in W^{1,r}(\Omega),\quad r\in[2,\infty).
\end{equation}
When $r>2$, the assumption that $\log\sigma\in W^{1,r}(\Omega)$ simply means 
that $\sigma\in W^{1,r}(\Omega)$ and  that $0<c<\sigma$
(strict ellipticity). If $r=2$,
then $\sigma$ lies in $W^{1,q}(\Omega)$ for all $q\in[1,2)$ by 
Proposition~\ref{embexpS}, but it is not necessarily bounded away from zero nor 
infinity which makes this case particularly interesting because 
\eqref{eq:u} may no longer be strictly elliptic.

Put $\nu:=(1-\sigma)/ (1+\sigma)$ and consider the
conjugate Beltrami equation:
\begin{equation}
\label{eq:f}
\overline\partial f  = \nu  \, \overline{\partial f} \quad\mbox{in } \Omega,
\quad -1\leq \nu\leq1, \quad \text{arctanh}\,\nu\in W^{1,r}(\Omega),\,\,\, 
r\in[2,\infty),
\end{equation}
where the assumptions on $\nu$ correspond to those on
$\sigma$
given in \eqref{eq:u}. The fact that
$\sigma\in W^{1,q}(\Omega)$ for all $q\in[1,2)$ implies easily that 
the same holds for $\nu$.  
If we restrict ourselves to solutions $f\in L^\gamma_{loc}(\Omega)$ for some
$\gamma>r/(r-1)$ and write  $f=u+iv$ to separate the real and the imaginary parts, 
we find that 
\eqref{eq:f} is equivalent to the generalized Cauchy--Riemann system:
\begin{equation} \label{system}
\left\{
\begin{array}{l}
\partial_xv=-\sigma\partial_y u,\\
\partial_yv=\sigma\partial_xu,
\end{array}
\right.
\end{equation}
whose compatibility condition is the conductivity equation
\eqref{eq:u}. Hence, \eqref{eq:f} is a means 
to rewrite \eqref{eq:u} as a complex equation of the first order.
Now, if we set 
$$
w:=\frac{f-\nu\overline{f}}{\sqrt{1-\nu^2}}=\sigma^{1/2}u+i\sigma^{-1/2}v,
\qquad \alpha=\bar\partial\log\sigma^{1/2}\in L^r,
$$
then a straightforward computation using \eqref{system}
shows that \eqref{eq:w} holds. Note that any constant $c$ solves
\eqref{eq:f}, the corresponding solution in \eqref{eq:w} being
$\sigma^{1/2}\text{Re}\,c+i\sigma^{-1/2}\text{Im}\,c$.

The preceding discussion makes
the study of \eqref{eq:f} essentially equivalent to that of \eqref{eq:w},
\eqref{eq:u}.
In particular, Theorem~\ref{Dirichlet2} translates into the following result 
that seems to be the first to describe a class of non strictly elliptic equations with 
unbounded coefficients for which the Dirichlet problem is well-posed with
(weighted) $L^p$-boundary data.

\begin{thm}
\label{Dirichletu}
Let $\sigma\geq0$ be such that $\log\sigma\in W^{1,2}(\D)$, and fix 
$p\in(1,\infty)$. For every $\psi$ such that $\psi\,\tr_\T\sigma^{1/2}\in L^p(\T)$,
there exists a unique solution $u$ to \eqref{eq:u} such that 
\begin{equation} \label{esssuppu}
\sup_{0<\rho  < 1}
\Bigl(\int_{\T_\rho} \left\vert
 u(\xi)\right\vert^p \sigma^{p/2}(\xi)|d\xi|\Bigr)^{1/p} <+\infty 
 \end{equation}
and $\lim_{\rho\to1} \tr_\T(u_\rho\sigma^{1/2}_\rho)=\psi\,\tr_\T\sigma^{1/2}$
in $L^p(\T)$. 
Moreover, the supremum in \eqref{esssuppu} is less than
$ C\|\psi\sigma^{1/2}\|_{L^p(\T)}$ for some $C=C(p,\sigma)$.
 \end{thm}
 
\section{Appendix}
\label{app}

\subsection{Mean growth of Cauchy transforms}
\label{croissC2}
In this subsection we prove estimate \eqref{estdRC2}.
First, we evaluate $\mathcal{C}_2(h)_{\D_R}$, the mean of $\mathcal{C}_2(h)$ over 
$\D_R$, when $h\in L^2(\C)$ and $R\geq1$.
To this end, we use the following identity (see \cite[Section 4.3.2]{aim}):
\begin{equation}
\label{Cauchycara}
\mathcal{C}(\chi_{\D_R})(t)=\left\{
\begin{array}{lcl}\bar t&\text{if}&|t|\leq R,\\
R^2/t&\text{if}& |t|>R.
\end{array}
\right.
\end{equation}
If $h$ has compact support, we deduce from
\eqref{defCauchy2}, \eqref{Cauchycara}  and  Fubini's theorem that
\begin{align}
\mathcal{C}_2(h)_{\D_R}&=
\frac{1}{\pi R^2}\int_{\D_R}\left(
\frac{1}{\pi}\int_\C \frac{h(t)}{z-t}\,dm(t)+
\frac{1}{\pi}\int_{\C\setminus\D}\frac{h(t)}{t}\,dm(t)
\right)dm(z) \nonumber\\
&=-\frac{1}{\pi R^2}\int_{\D_R} h(t)\bar t\,dm(t)
-\frac{1}{\pi}\int_{\C\setminus\D_R} \!\frac{h(t)}{t}\,dm(t)
+\frac{1}{\pi}\int_{\C\setminus\D}\frac{h(t)}{t}\,dm(t)
\nonumber\\
&=-\frac{1}{\pi R^2}\int_{\D_R} h(t)\bar t\,dm(t)
+\frac{1}{\pi}\int_{1\leq|t|\leq R}\frac{h(t)}{t}\,dm(t).
\label{meanh}
\end{align}
By density argument, \eqref{meanh} holds for every $h\in L^2(\C)$.
Next, by \eqref{meanh} and the Schwarz inequality, we have
\begin{equation}
\label{estmeanpsi}
\left|\mathcal{C}_2(h)_{\D_R}\right|\leq
\frac{\|h\|_{L^2(\C)}}{\sqrt{2\pi}}
+\sqrt{\frac{2}{\pi}}\,\|h\|_{L^2(\C)}\left(\log R\right)^{1/2},
\qquad R\geq1.
\end{equation}
In another connection, by the Poincar\'e  inequality, we have
\begin{multline}
\label{ecartpsimoy}
\|\mathcal{C}_2(h)-\mathcal{C}_2(h)_{\D_R}\|_{L^2(\D_R)}
\leq C_R\bigl(\|h\|_{L^2(\D_R)}
+\|\mathcal{B}(h)\|_{L^2(\D_R)}\bigr)\\ \le 2C_R\|h\|_{L^2(\C)},
\end{multline}
where $C_R$ is the best constant in \eqref{estfonc} for $p=2$ and
$\Omega=\D_R$. Finally, 
since
$$
\frac{\|\mathcal{C}_2(h)\|_{L^2(\D_R)}}{\sqrt{\pi}R}\leq 
\frac{\|\mathcal{C}_2(h)-\mathcal{C}_2(h)_{\D_R}\|_{L^2(\D_R)}}{\sqrt{\pi}R}
+\left|\mathcal{C}_2(h)_{\D_R}\right|,
$$
\eqref{estdRC2} follows from
\eqref{estmeanpsi}, \eqref{ecartpsimoy} and the fact that
$C_R=RC_1$ by homogeneity.

\subsection{Functions of vanishing  mean oscillation}
The space $BMO(\T)$ of functions with bounded mean
oscillation on the unit circle  consists of the functions  
$h\in L^1(\T)$ such that
\begin{multline}
\label{defBMOT}
\|h\|_{BMO(\T)}:=
\sup_{I}\frac{1}{\Lambda(I)}\int_{I}|h(t)-h_{I}|
\,d\Lambda(t)<\infty,\\ h_I:=\frac{1}{\Lambda(I)}\int_Ih(t)\,d\Lambda(t),
\end{multline}
where $\Lambda$  indicates arclength and $I$ ranges over all subarcs 
of $\T$. 
Note that $\|\cdot\|_{BMO(\T)}$ is a genuine norm 
modulo additive constants only. 
The space $VMO(\T)$ of functions with vanishing mean oscillation is
the subspace of $BMO(\T)$ consisting of those $h$ for which
\begin{equation}
\label{defVMOT}\lim_{\varepsilon\to0} \,\sup_{\Lambda(I)<\varepsilon}\frac{1}{\Lambda(I)}
\int_{I}|h(t)-h_{I}|
\,d\Lambda(t)=0.
\end{equation}
Actually, $VMO(\T)$ is the closure in $BMO(\T)$ of continuous 
functions \cite[Chapter VI, Corollary 1.3 \& Theorem 5.1]{gar}.
The John--Nirenberg theorem 
asserts that there exist absolute constants $C$, $c$ such that, 
for every $h\in BMO(\T)$, every arc $I\subset\T$, and any $\lambda>0$,
\begin{equation}
\label{JNT}
\frac{\Lambda(\{\xi\in I~:~|h(\xi)-h_I|>\lambda\}  )}{\Lambda(I)}\leq 
C\exp\left(\frac{-c\lambda}{\|h\|_{BMO(\T)}}\right);
\end{equation}
in fact one can take $C=e$ and $c=1/2e$, see  \cite[Theorem 7.1.6]{Grafakos}\footnote{The argument there is given on the line but it applies mutatis mutandis to the circle.}. We also need a quantitative version of the so-called integral 
form of the John--Nirenberg inequality\footnote{When $M_h(I)$ 
gets replaced by 
$\sup_{I'\subset I}\left(\frac{1}{\Lambda(I')}\int_{I'}|h-h_{I'}|^2\,d\Lambda\right)^{1/2}$ (a different but in fact equivalent quantity), the sharp constants in 
\eqref{espsI} were obtained in
\cite{SV}.}.
Given $h\in L^1(\T))$ and an arc $I\subset\T$, let us define
\begin{equation}
\label{defMhG}
M_h(I):=\sup_{I'\subset I}\frac{1}{\Lambda(I')}\int_{I'}|h-h_{I'}|\,d\Lambda,
\end{equation}
where the supremum is taken over all subarcs $I'\subset I$.

\begin{lem}
\label{expsumVMOG}
If $h\in BMO(\T)\setminus\{0\}$
and $I\subset\T$ is an arc, then
\begin{equation}
\label{espsI}
\int_Ie^{|h|/(4eM_h(I))}d\Lambda\leq (1+e)\Lambda(I)\,\,e^{|h_I|/(4eM_h(I))}.
\end{equation}
\end{lem}

\begin{proof} 
Inspecting  the standard proof of the John--Nirenberg inequality that uses recursively the Calder\`on--Zygmund decomposition 
on dyadic subdivisions of $I$ \cite[Chapter VI, Theorem 2.1]{gar}, one checks that 
\eqref{JNT} remains valid  if we replace $\|h\|_{BMO(\T)}$  by $M_h(I)$:
\begin{equation}
\label{JNQMT}
\frac{\Lambda\bigl(\{\xi\in I~:~|h(\xi)-h_I|>\lambda\}\bigr)}{\Lambda(I)}\leq 
C\exp\left(\frac{-c\lambda}{M_h(I)}\right).
\end{equation}
Pick $c'\in(0,c)$ with $c$  as in \eqref{JNQMT}, and set 
$g:=c'|h-h_I|/M_h(I)$. We compute  as in  \cite[Corollary 7.1.7]{Grafakos}:
\begin{multline*}
\frac{1}{\Lambda(I)}\int_Ie^g\,d\Lambda=1+\frac{1}{\Lambda(I)}\int_I(e^g-1)\,d\Lambda\\=1+\frac{1}{\Lambda(I)}\int_0^\infty
e^\lambda\Lambda(\{\xi\in I\,:\,g(\xi)>\lambda\})\,d\lambda
\end{multline*}
where the second equality follows from Fubini's theorem.
Using \eqref{JNQMT} to estimate the distribution function of $g$, 
we find that
\begin{multline*}
\frac{1}{\Lambda(I)}\int_Ie^{c'|h-h_I|/M_h(I)} d\Lambda=\frac{1}{\Lambda(I)}\int_Ie^g\,d\Lambda\\
\le 1+C\int_0^\infty e^\lambda e^{-c\lambda /c'}d\lambda=1+\frac{C}{c/c'-1}.
\end{multline*}
Choosing 
$C=e$, $c=1/(2e)$, and $c'=1/4e$, we obtain
\begin{equation}
\label{estdf1}
\frac{1}{\Lambda(I)}\int_Ie^{|h-h_I|/(4eM_h(I))} d\Lambda\leq
1+e
\end{equation}
from which  \eqref{espsI} follows at once.
\end{proof}

By definition, $M_h(I)$ tends to zero uniformly with
$\Lambda(I)$ if  $h\in VMO(\T)$, and Lemma~\ref{expsumVMOG}
makes it clear that in this case $e^h\in L^p(\T)$ for every $p\in[1,\infty)$. 

When $h\in VMO_\R(\T)$, where subscript
``$\R$'' means ``real-valued'' as usual,
it is well known  that $e^h$
satisfies  
condition $A_p$ given in \eqref{defAp} for all $p\in(1,\infty)$. 
This follows for instance 
from \eqref{estdf1} and \cite[Chapter VI, Corollary 6.5]{gar}.
Below, we record for 
later use a specific  estimate for the $A_p$ norm in terms of \eqref{defMhG}.

\begin{lem}
\label{ApVMO}
Let $h\in VMO_\R(\T)$ and $p\in(1,\infty)$. Let $\eta=\eta(h,p)>0$ be so 
small that $4eM_h(I)\max(1,1/(p-1)\bigr)\leq1$ for every arc $I\subset\T$ satisfying
$\Lambda(I)<\eta$. 
Then
\begin{equation}
\label{estAp}
\{e^h\}_{A_p}:=\sup_{I}\left(\frac{1}{\Lambda(I)}\int_{I} e^h\,d\Lambda\right)
\left(\frac{1}{\Lambda(I)}\int_I e^{-h/(p-1)} d\Lambda\right)^{p-1}
\leq C
\end{equation}
where $C$ depends only on $\eta$, $p$, and $\|e^h\|_{L^1(\T)}$.
\end{lem}

\begin{proof} If we put $p'=p/(p-1)$, then $1/(p-1)=p'-1$ and it follows
easily from the definition that 
$\{e^h\}_{A_p}= \{e^{-h/(p-1)}\}_{A_{p'}}^{(p-1)}$.
Therefore we may assume that $p\geq2$. 

Now, the left hand side of \eqref{estAp} can be rewritten as
$$
\sup_{I}\Bigl(\frac{1}{\Lambda(I)}\int_{I} e^{h-h_I}\,d\Lambda\Bigr)
\Bigl(\frac{1}{\Lambda(I)}\int_I e^{-(h-h_I)/(p-1)} d\Lambda\Bigr)^{p-1}.
$$
If $\Lambda(I)<\eta$, then $4eM_h(I)$ and $4eM_h(I)/(p-1)<1$, thus
by \eqref{estdf1} and H\"older's inequality we have 
\begin{multline*}
\frac{1}{\Lambda(I)}\int_{I} e^{h-h_I}\,d\Lambda\le
\Bigl(\frac{1}{\Lambda(I)}\int_{I} e^{|h-h_I|/(4eM_h(I))}\,
d\Lambda\Bigr)^{4eM_h(I)}\\ \le(1+e)^{4eM_h(I)}\le(1+e)
\end{multline*}
and
\begin{multline*}
\Bigl(\frac{1}{\Lambda(I)}\int_{I} e^{-(h-h_I)/(p-1)}\,d\Lambda\Bigr)^{p-1}
\le\Bigl(\frac{1}{\Lambda(I)}\int_{I} e^{|h-h_I|/(4eM_h(I))}\,
d\Lambda\Bigr)^{4eM_h(I)}\\ \le(1+e).
\end{multline*}
This shows that \eqref{estAp} holds with $C=(1+e)^2$ when the supremum is 
restricted to those $I$ of length less than $\eta$.
In another connection, if $\Lambda(I)\geq\eta$, then obviously
\[\frac{1}{\Lambda(I)}\int_{I} e^h\,d\Lambda\leq
\eta^{-1}\|e^{|h|}\|_{L^1(\T)}\]
and likewise, taking into account that $p\geq2$ and
using H\"older's inequality, we obtain
\[
\left(\frac{1}{\Lambda(I)}\int_{I} e^{-h/(p-1)}\,d\Lambda\right)^{p-1}
\leq\frac{1}{\Lambda(I)}\int_{I} e^{|h|}\,d\Lambda\leq
\eta^{-1}\|e^{|h|}\|_{L^1(\T)}.
\]
Thus, \eqref{estAp} holds with $C=\bigl(\|e^{|h|}\|_{L^1(\T)}/\eta\bigr)^2$ 
in this case.
\end{proof}

When $\Gamma$ is a Jordan curve locally isometric to a Lipschitz graph, the definitions of  $BMO(\Gamma)$,
$VMO(\Gamma)$, and condition $A_p$ on $\Gamma$ which are modeled after 
\eqref{defBMOT}, \eqref{defVMOT}, and \eqref{defAp} do coincide with 
the standard ones \cite[Section 2.5]{BoKa}\footnote{In the standard definition, 
arcs $I\subset\Gamma$
are replaced by sets of  type $\D(\xi,\rho)\cap\Gamma$ with $\xi\in\Gamma$.
It is in this form that condition $A_p$ is necessary and sufficient for
weighted $L^p$ boundedness of the singular Cauchy integral operator
on $\Gamma$, see 
\cite[Chapter 5]{BoKa}.}. 
Lemma~\ref{expsumVMOG} and Lemma~\ref{ApVMO}  
carry over mechanically to this 
more general setting, but the significance of condition $A_p$ 
with respect to the weighted $L^p$ continuity of the conjugate operator 
is no longer the same if  $\Gamma$ is non-smooth\footnote{
Even if we restrict ourselves to
\emph{constant} weights (which certainly satisfy $A_p$ for all 
$p\in(1,\infty)$), the conjugate operator  is generally $L^p$-continuous 
for  restricted range of $p$ only. This follows from
\cite[Theorem 2.1]{LS} and the fact that the Szeg\H{o} projection has the same 
weighted $L^p$ type as the conjugate operator on $\T$.}. 
Such considerations are not needed in this paper, but we
make use at some point of the following estimate showing that
$W^{1/2,2}(\Gamma)$ embeds contractively 
in $VMO(\Gamma)$ \cite{BrN1}:
\begin{equation}
\label{VMOW1/2}
\begin{array}{lll}
\frac{1}{\Lambda(I)}\int_{I}|h-h_{I}|\,d\Lambda&\leq& \frac{1}{(\Lambda(I))^2}
\int_{I\times I}|h(t)-h(t')|\,d\Lambda(t)d\Lambda(t')\\
&\leq& \frac{1}{\Lambda(I)}
\int_{I\times I}\frac{|h(t)-h(t')|}{\Lambda(t,t')}
d\Lambda(t)d\Lambda(t')\\
&\leq&
\left(
\int_{I\times I}\frac{|h(t)-h(t')|^2}{(\Lambda(t,t'))^2}d\Lambda(t)d\Lambda(t')\right)^{1/2}\\
&\leq& 
 \|h\|_{W^{1/2,2}(\Gamma)},
\end{array}
\end{equation}
where the next to last step uses the Schwarz inequality.
Note that if $h\in W^{1/2,2}(\Gamma)$, then 
$$
\|h\|_{W^{1/2,2}(I)}:=
\left(\int_{I\times I}\frac{|h(t)-h(t')|^2}{(\Lambda(t,t'))^2}d\Lambda(t)d\Lambda(t')\right)^{1/2}
$$ 
tends to $0$ as $\Lambda(I)\to 0$ 
by the absolute continuity of $\frac{|h(t)-h(t')|^2}{(\Lambda(t,t'))^2}d\Lambda(t)d\Lambda(t')$.

\subsection{Exp-summability of Sobolev functions at the critical exponent}
Given   a bounded open set $\Omega\subset\C$,
the Trudinger-Moser inequality \cite{Moser} asserts that
\begin{equation}
\label{inegMT}
\sup_{\stackrel{h\in W^{1,2}_0(\Omega)}{\|\partial h\|^2_{L^2(\Omega)}+
\|\bar \partial h\|^2_{L^2(\Omega)}\leq1/2}}\,\int_\Omega  e^{4\pi |h|^2}dm\leq C_{\text{TM}}
|\Omega|
\end{equation}
for some absolute constant $C_{\text{TM}}$. Now, given a nonzero 
$f\in W^{1,2}_0(\Omega)$, put for the sake of simplicity 
$N_1(f):=(2\|\partial f\|^2_{L^2(\Omega)}+
2\|\bar \partial f\|^2_{L^2(\Omega)})^{1/2}$ and let further $f_1=f/N_1(f)$. 
For each
$\xi\in\Omega$ such that $f(\xi)$ is defined, we have either
$|f(\xi)|\leq N_1^2(f)/4\pi$ 
or $\exp(|f(\xi)|)< \exp(4\pi|f_1(\xi)|^2)$.
Thus, applying \eqref{inegMT} with  $h=f_1$, we obtain for $f\in W^{1,2}_0(\Omega)$ {\it a fortiori} that
\begin{equation}
\label{inegMTs}
\int_\Omega e^{|f|}dm\leq|\Omega|\Bigl(C_{\text{TM}}+
\exp\Bigl({\frac{\|\partial f\|^2_{L^2(\Omega)}+
\|\bar \partial f\|^2_{L^2(\Omega)}}{2\pi}}\Bigr)\Bigr).
\end{equation}

\begin{lem}
\label{LpexpW12}
Let $\Omega\subset\C$ be a bounded and Lipschitz  open set.
Then there exist $C_1=C_1(\Omega)$, $C_2=C_2(\Omega)$ such that, for every $\ell\in[1,\infty)$
and $f\in W^{1,2}(\Omega)$,
\begin{equation}
\label{estimee normexp}
\|e^{|f|}\|_{L^\ell(\Omega)}\leq C_1
\exp(C_2\ell \|f\|^2_{W^{1,2}(\Omega)}).
\end{equation}
\end{lem}

\begin{proof}
Let $\Omega_1\supset\overline{\Omega}$ be open and, say 
$|\Omega_1|\leq 2|\Omega|$.
Pick $\varphi\in\mathcal{D}(\R^2)$ to have 
support in $\Omega_1$,  values in $[0,1]$, and 
to be identically 1 on   $\Omega$. 
By the extension theorem, there exists 
$\tilde{f}\in W^{1,2}(\R^2)$ such that $\tilde{f}_{|\Omega}=f$ and
$\|\tilde{f}\|_{W^{1,2}(\R^2)}\leq C\|f\|_{W^{1,2}(\Omega)}$,
where $C=C(\Omega)$. Then $h:=\ell\varphi\tilde{f}$ lies in 
$W^{1,2}_0(\Omega_1)$  and satisfies
$$
\|\partial h\|^2_{L^2(\Omega_1)}+\|\bar \partial h\|^2_{L^2(\Omega_1)}\le \ell^2 C'\|f\|^2_{W^{1,2}(\Omega)},
$$
where $C'$ depends on $C$ and $\varphi$. 
Applying \eqref{inegMTs} to $h$,
we find on putting $C_2=C'/(2\pi)$ that
$$
\int_{\Omega}e^{\ell|f|}dm\leq \int_{\Omega_1}e^{|h|}dm\leq
\left(e^{\ell^2 C_2\|f\|^2_{W^{1,2}(\Omega)}}+C_{\text{TM}}\right)|\Omega_1|,
$$
that yields \eqref{estimee normexp} upon
setting $C_1:=2(1+C_{\text{TM}})|\Omega|$.
\end{proof}

With the help of  Lemma~\ref{LpexpW12}, we now prove that
$e^f$ is fairly
 smooth 
when $f\in W^{1,2}(\Omega)$. Recall that a (possibly nonlinear) operator 
between Banach spaces is said  to be bounded if it maps bounded sets into 
bounded sets. 

\begin{prop}
\label{embexpS}
Let $\Omega\subset\R^2$ be a bounded Lipschitz smooth open set. Fix
$p\in(1,\infty)$ and   $\ell\in[1,\min(p,2))$.
Then, the map $(g,f)\mapsto ge^f$ is continuous and bounded from
$W^{1,p}(\Omega)\times W^{1,2}(\Omega)$ into
$W^{1,\ell}(\Omega)$,  and derivatives are  computed using the Leibniz and
the chain rules: 
\begin{equation}
\label{CR2}
\partial(ge^f)=e^f\partial g +ge^f\partial f,\qquad
\bar\partial(ge^f)=e^f\bar\partial g +ge^f\bar\partial f.
\end{equation}
In particular, for every $q\in[1,2)$, the map $f\mapsto e^f$
is continuous and bounded from $W^{1,2}(\Omega)$ into $W^{1,q}(\Omega)$ and
so is the map $f\mapsto e^{\tr_{\partial\Omega}f}$  
from $W^{1,2}(\Omega)$ into $W^{1-1/q,q}(\partial\Omega)$.
\end{prop}

\begin{proof} 
Let  $g\in W^{1,p}(\Omega)$,  $f\in W^{1,2}(\Omega)$, and 
let $(f_n)$, $(g_n)$ be two sequences of smooth functions on $\Omega$ 
converging respectively to $f$ and $g$ in $W^{1,2}(\Omega)$ and 
$W^{1,p}(\Omega)$.
\emph{We claim} that $e^{f_n}$ converges to $e^f$ in $L^\ell(\Omega)$ for all 
$\ell\in[1,\infty)$. To see this, consider first the case of real-valued 
functions. By the mean-value theorem and convexity 
of $t\mapsto e^t$, we have that
\begin{multline}
\label{manqt}
\int_\Omega \bigl|e^f-e^{f_n}\bigr|^\ell\,dm\leq
\int_\Omega |f-f_n|^\ell\bigl|e^f+e^{f_n}\bigr|^\ell\,dm\\ \le
\|f-f_n\|_{L^{2\ell}(\Omega)}^{\ell}\|e^f+e^{f_n}\|_{L^{2\ell}(\Omega)}^\ell,
\end{multline}
where we use the Schwarz inequality. By the Sobolev embedding theorem, 
$\|f-f_n\|_{L^{2\ell}(\Omega)}$ tends to $0$ as $n\to\infty$. 
Moreover, $\|f_n\|_{W^{1,2}(\Omega)}$ tends to $\|f\|_{W^{1,2}(\Omega)}$, 
hence $\|e^f+e^{f_n}\|_{L^{2\ell}(\Omega)}$ is uniformly bounded
by Lemma~\ref{LpexpW12}, and the right hand side of \eqref{manqt}
indeed goes to zero as $n\to\infty$.
Next, if  $f$, $f_n$ are complex-valued, say $f=u+iv$ and $f_n=u_n+iv_n$,
we write 
$$
\|e^f-e^{f_n}\|_{L^\ell(\Omega)}\leq\|e^u(e^{iv}-e^{iv_n})\|_{L^\ell(\Omega)}
+\|e^{iv_n}(e^u-e^{u_n})\|_{L^\ell(\Omega)}.
$$
By what precedes, the last term in the right hand side tends to $0$
 when $n\to\infty$, and so does the first since we can extract  
pointwise  convergent subsequences from 
any subsequence of $v_n$ and apply the dominated convergence theorem.
This proves the claim.

Next, we observe that $g_n e^{f_n}$ is smooth on $\Omega$ and that
\begin{equation}
\label{CRS2}
\partial(g_ne^{f_n})=e^{f_n}\partial g_n +g_ne^{f_n}\partial f_n.
\end{equation}
Assume first that $p<2$.
Then, by  the Sobolev embedding 
theorem, $(g_n)$ converges to $g$ in $L^{p^*}(\Omega)$ where $p^*=2p/(2-p)>2$.
From this and the previous claim, we deduce by H\"older's inequality
that $(g_ne^{f_n})$ converges to 
$ge^f$ in $L^\ell(\Omega)$ for $\ell\in [1,p^*)$, hence also in the sense of  
distributions.
By the same token, the right hand side of \eqref{CRS2} converges to 
$e^f\partial g +ge^f\partial f$ in $L^\ell(\Omega)$ for each 
$\ell\in[1,p)$. The case $p=2$ is similar except that $p^*$ can be 
taken arbitrarily large, hence the convergence in the right 
hand-side of \eqref{CRS2} takes place in $L^\ell(\Omega)$ for  all $\ell<2$.
If $p>2$,  then $g$ is even bounded, but this does not improve the estimate.
Repeating the argument for  $\bar\partial(ge^f)$ proves that 
$(g_n e^{f_n})$ converges to $g e^f$ in $W^{1,\ell}$ for $\ell\in[1,\min(p,2))$
and that \eqref{CR2} holds. Hence, the map $(g,f)\mapsto g e^f$ is defined from 
$W^{1,p}(\Omega)\times W^{1,2}(\Omega)$ into
$W^{1,\ell}(\Omega)$ and \eqref{CR2} is valid. Moreover, by
Lemma~\ref{LpexpW12} and H\"older's inequality, this map is bounded.
Relaxing the smoothness assumption on $f_n$, $g_n$ 
and arguing as before shows that it is also continuous.
This proves the first assertion on the Proposition.
Setting $g\equiv1$, the  second assertion follows by 
the Sobolev embedding and the trace theorems.
\end{proof}

\subsection{Equicontinuity properties of Cauchy transforms}

\begin{prop}
\label{evalderCauchy}
Let $\beta\in L^2(\D)$
and let $\omega$  be a strictly
positive function on $(0,+\infty)$ such 
that
$\|\beta\|_{L^2(Q_{\omega(\varepsilon)}\cap\D)}<\varepsilon$ as soon as 
$Q_{\omega(\varepsilon)}$ is a square  of sidelength $\omega(\varepsilon)$. 
\begin{itemize}
\item[(i)] If we set {\rm(}{\it cf.} \eqref{defCauchy}{\rm)}
$$
\mathcal{C}(\beta)(z)=
\frac{1}{2\pi i}\int_{\D}\frac{\beta(\xi)}{\xi-z}\,d\xi\wedge\overline{d\xi},
\qquad z\in\C,
$$
then there exists
a strictly positive function  $\omega_1$ on  $(0,+\infty)$,
depending only on $\omega$, such that
\begin{equation}
\label{borderC}
\|\partial \mathcal{C}(\beta)\|_{L^2(Q_{\omega_1(\eta)})}+
\|\bar \partial \mathcal{C}(\beta) \|_{L^2(Q_{\omega_1(\eta)})}
<\eta
\end{equation}
as soon as 
$Q_{\omega_1(\eta)}$ is a square of sidelength $\omega_1(\eta)$. 
\item[(ii)] If we set  {\rm(}{\it cf.} \eqref{normreelR}{\rm)}
$$
\mathcal{R}(\beta)(z):=\frac{1}{2\pi i}\int_{\D}\frac{z\bar \beta(\xi)}{1-\bar \xi z }
\,d\xi\wedge\overline{d\xi},
\qquad z\in\D,
$$
then $\mathcal{R}(\beta)\in W^{1,2}(\D)$ is holomorphic in $\D$ and there exists
a strictly positive function  $\omega_2$ on $(0,+\infty)$,
depending only on $\omega$, such that
$$
\|\partial \mathcal{R}(\beta)\|_{L^2(Q_{\omega_2(\eta)}\cap\D)}<\eta
$$
as soon as 
$Q_{\omega_2(\eta)}$ is a square of sidelength $\omega_2(\eta)$. 
\end{itemize}
\end{prop}

\begin{proof} Since $\beta\in L^2(\D)$, we know that 
$\mathcal{C}(\beta)\in W^{1,2}_{loc}(\C)$.
Fix $\eta>0$ and set $\delta=\min(1/3,\omega(\eta/3),\eta/(6\|\beta\|))$.
For any square $Q_{\delta}$, we have a nested 
concentric square with parallel sides
$Q_{\delta^2}\subset Q_{\delta}$.  
Let  $\widetilde{\beta}$ be the extension of $\beta$ by 0 off $\D$.
Since $\bar\partial\mathcal{C}(\beta)=\widetilde{\beta}$, we obtain 
\begin{equation}
\label{bsdbar}
\|\bar \partial \mathcal{C}(\beta) \|_{L^2(Q_{\delta^2})}
<\eta/3.
\end{equation}
Next, we write
\begin{equation}
\label{formd}
\partial\mathcal{C}(\beta)=\mathcal{B}(\widetilde{\beta})=
\mathcal{B}(\chi_{Q_{\delta}}\widetilde{\beta})+
\mathcal{B}(\chi_{\CC\setminus Q_{\delta}}\widetilde{\beta})
\end{equation}
where $\mathcal{B}$ indicates the Beurling transform, {\it cf.} \eqref{defB}.
As  $\mathcal{B}$ is an isometry on $L^2(\CC)$, we get
\begin{equation}
\label{bsd}
\|\mathcal{B}(\chi_{Q_{\delta}}\widetilde{\beta})\|_{L^2(\CC)}=
\|\beta\|_{L^2(Q_{\delta}\cap\DD)}
<\eta/3.
\end{equation}
Moreover, formula \eqref{defB} and the Cauchy-Schwarz inequality 
give us the pointwise estimate:
$$
\mathcal{B}(\chi_{\CC\setminus Q_{\delta}}\widetilde{\beta})(z)
\le \frac2\delta \|\beta\|_{L^2(\DD)},
\qquad z\in Q_{\delta^2}.
$$
Integrating over $Q_{\delta^2}$ yields
\begin{equation}
\label{bsd1}
\|\mathcal{B}(\chi_{\CC\setminus Q_{\delta}}\widetilde{\beta})
\|_{L^2(Q_{\delta^2})}\le 
2\delta\|\beta\|_{L^2(\DD)} 
<\eta/3.
\end{equation}
Inequality \eqref{borderC} with $\omega_1(\eta)=\delta$ follows now from \eqref{bsdbar}, \eqref{formd}, \eqref{bsd}
and \eqref{bsd1}, thereby proving $(i)$.

Consider next $\mathcal{R}(\beta)$. Clearly it is holomorphic in $\D$ and
vanishes at $0$. Furthermore, 
$\overline{\mathcal{R}(\beta)(z)}=-\mathcal{C}(\beta)(1/\bar z)$ and since
$\mathcal{C}(\beta)\in W^{1,2}_{loc}(\CC)$ 
we get that $\mathcal{R}(\beta)\in W^{1,2}(\DD)$.

Once again, fix $\eta>0$ and set $\delta=\min(\omega_1(\eta)/4,\eta/(16\|\beta\|))$. 
First, every square
$Q_{\delta}$ has diameter at most $1/4$, hence is disjoint from
$\DD_{1/2}$ if it meets $\mathcal{A}_{3/4}:=\{z:\ 1\geq|z|\geq3/4\}$. 
In this case the
reflection ($z\mapsto 1/\bar z$) of $Q_{\delta}\cap\DD$ is contained in a square of 
sidelength $4\delta\le\omega_1(\eta)$, and since
$$
\partial \mathcal{R}(\beta)(z)=
\frac{\overline{ (\partial(\mathcal{C}\beta))(1/\bar z)}}{ z^2},\qquad z\neq0,
$$
we deduce from \eqref{borderC} and the change of variable formula that 
$\|\partial\mathcal{R}(\beta)\|_{L^2(Q_{\delta})}\allowbreak\le
\eta $.

Assume now that $Q_{\delta}\subset \DD_{3/4}$.
Differentiating under the integral sign we obtain
$$
\partial \mathcal{R}(\beta)=
\frac{1}{2\pi i}\int_{\D}\frac{\bar \beta(\xi)}{1-\bar \xi z}\,d\xi\wedge\overline{d\xi}+
\frac{1}{2\pi i}\int_{\D}\frac{z\bar\xi\bar \beta(\xi)}{(1-\bar \xi z)^2 }\,d\xi\wedge\overline{d\xi},
$$
so that if $z\in\DD_{3/4}$, we get by the Schwarz inequality that
$|\partial \mathcal{R}(\beta)(z)|\leq16\|\beta\|_{L^2(\DD)}$.
Integrating over $Q_{\delta}$ yields
$$
\|\partial\mathcal{R}(\beta)\|_{L^2(Q_{\delta})}\le
16\delta\|\beta\|_{L^2(\DD)}\le\eta,
$$
as desired. It remains to set $\omega_2(\eta)=\delta$.
\end{proof}

\begin{cor}
\label{equicontC2}
Let $\beta\in L^2(\C)$ and let $\omega$ be a strictly positive function on $(0,+\infty)$ such that
$\|\beta\|_{L^2(Q_{\omega(\varepsilon)})}<\varepsilon$ as soon as 
$Q_{\omega(\varepsilon)}$ is a square of sidelength $\omega(\varepsilon)$. 
If we let {\rm(}cf. \eqref{defCauchy2}{\rm)}
$$
\mathcal{C}_2(\beta)(z)=\frac{1}{\pi}\int_{\R^2}\Bigl(\frac{1}{z-t}+
\frac{\chi_{\C\setminus\D}(t)}{t}\Bigr) \beta(t)\,dm(t),
\qquad z\in\C,
$$
then there exists a strictly positive function  $\omega_1$ on  $(0,+\infty)$,
depending only on $\omega$, such that
$$
\|\partial \mathcal{C}_2(\beta)\|_{L^2(Q_{\omega_1(\eta)})}+
\|\bar \partial \mathcal{C}_2(\beta) \|_{L^2(Q_{\omega_1(\eta)})}<\eta
$$
as soon as 
$Q_{\omega_1(\eta)}$ is a square of sidelength $\omega_1(\eta)$. 
\end{cor}

\begin{proof} This is proved in the same way as \eqref{borderC},
replacing $\tilde{\beta}$ by $\beta$.
\end{proof}

\subsection{Integral estimates on circular arcs}

\begin{lem}
\label{bornewr}
If $f\in W^{1,q}(\D)$ for some $q\in(1,2)$ and 
$\ell:= q/(2-q)$, then
$$
\sup_{\rho\in(0,1]}\Bigl(\int_{\T_\rho}|f(\xi)|^{\ell}|d\xi|\Bigr)^{1/\ell}\le C\|f\|_{W^{1,q}(\D)}, 
$$
where $C=C(q)$.
\end{lem}

\begin{proof}
Set $f_\rho(\xi):=f(\rho\xi)$ so that
\begin{equation}
\label{transfTi}
\Bigl(\int_{\T_\rho}|f(\xi)|^{\ell}|d\xi|\Bigr)^{1/\ell}=
\rho^{1/\ell}\Bigl(\int_{\T}|f_\rho(\xi)|^{\ell}|d\xi|\Bigr)^{1/\ell}.
\end{equation}
By the trace theorem and (the non-integral version of) the Sobolev embedding theorem we have 
\begin{equation}
\label{niS}
\Bigl(\int_{\T}|f_\rho(\xi)|^{\ell}|d\xi|\Bigr)^{1/\ell}\leq C
\|f_\rho\|_{W^{1,q}(\D)}
\end{equation}
with $C=C(q)$,
and from the  change of variable formula we get for $\rho>0$ that
\begin{equation}
\label{estr}
\|f_\rho\|_{W^{1,q}(\D)}=\rho^{-2/q}\|f\|_{L^q(\D_\rho)}+
\rho^{1-2/q}
\left(
\|\partial f\|_{L^q(\D_\rho)}+\|\bar \partial f\|_{L^q(\D_\rho)}
\right).
\end{equation}
Since $1/\ell-2/q=-1$, and in view of \eqref{transfTi}, \eqref{niS},
and \eqref{estr} it remains to majorize 
$\rho^{-1}\|f\|_{L^q(\D_\rho)}$ by
$C\|f\|_{W^{1,q}(\D)}$ for some $C=C(q)$.
From \eqref{homQprim} 
we see that this is equivalent to checking the estimate:
\begin{equation}
\label{inegmoy}
\rho^{2/q-1}\left|f_{\D_\rho}\right|=\Bigl|\frac{1}{\pi \rho^{3-2/q}}
\int_{\D_\rho}f\,dm\Bigr|\le C_1\left\|f\right\|_{W^{1,q}(\D)},
\qquad 0<\rho\leq1,
\end{equation}
with $C_1=C_1(q)$. Now, the Sobolev embedding theorem implies that for some  
$C_2=C_2(q)$ we have 
$\|f\|_{ L^{2q/(2-q)}(\D)}\leq C_2\|f\|_{W^{1,q}(\D)}$,
and so by H\"older's inequality,
$$
\Big|\int_{\D_\rho}f\,dm\Bigr|
\le C_2\pi^{3/2-1/q}\rho^{3-2/q}\|f\|_{W^{1,q}(\D)}
$$
which is exactly \eqref{inegmoy} with $C_1=C_2\pi^{1/2-1/q}$. 
\end{proof}

For $J\subset\T$ an open arc and $\delta\in(0,1)$,
we denote by $R(J,\delta)$  
the open curvilinear rectangle in $\D$  (an annulus if $J=\T$) defined by
\begin{equation}
\label{defR}
R(J,\delta)=
\{z:\ z=\rho\xi,\ \xi\in J, \ 1-\delta<\rho<1\}.
\end{equation}

\begin{lem}
\label{moyW120}
If $f\in W^{1,2}_0(\D)$ and $\rho\in(0,1]$,
then for every arc $I\subset\T_\rho$ we have 
\begin{multline}
\label{estmoy}
\Bigl\vert\frac{1}{\Lambda(I)}\int_{I}f(\zeta) \,|d\zeta|\Bigr\vert \\ \le
\frac{(1-\rho)^{1/2}}{(\Lambda(I))^{1/2}}
\,\Bigl(\|\partial f\|_{L^2(R(J,1-\rho)}+\|\bar \partial f\|_{L^2(R(J,1-\rho)}\Bigl),
\end{multline}
where $J\subset\T$ is the arc such that $\rho J=I$.
\end{lem}

\begin{proof} By density it suffices to 
prove \eqref{estmoy} when $f\in\mathcal{D}(\D)$.  If we write 
$\zeta\in I$ as  $\zeta=\rho\xi$ with $\xi\in J$, we get
$$
f(\zeta)=-\int_\rho^1 \left(\partial f(t\xi)\xi+\bar \partial f(t\xi)
\bar\xi\right)\,dt
$$
and integrating with respect to $|d\zeta|=\rho|d\xi|$ yields
\begin{multline*}
\Bigl\vert\int_{I} f(\zeta) \,|d\zeta|\Bigr\vert=
\Bigl\vert\rho\int_{J}\int_\rho^1  \left(\partial f(t\xi)\xi+\bar \partial f(t\xi)\right)\,dt
|d\xi|\Bigr\vert\\ 
\le \int_{R(J,1-\rho)}  \Bigl(\left\vert\partial f(t\xi)\right\vert+
\left\vert\bar\partial f(t\xi)\right\vert\Bigr)\,tdt|d\xi|.
\end{multline*}
Since $m(R(J,1-\rho))=\Lambda(I)(1-\rho^2)/2$, estimate
\eqref{estmoy} follows from the Schwarz inequality.
\end{proof}

\begin{lem}
\label{ecmoyW120}
Let $J$ be a proper open subarc of $\T$ and
let $\delta_0\in(0,1)$. 
For every $\delta\in(0,\delta_0]$ there exists $C>0$
depending only on $\delta_0$ and $\Lambda(J)/\delta$ such that,
for all $f\in W^{1,2}(R(J,\delta))$ {\rm(}{\it cf.} definition \eqref{defR}{\rm)} we have 
\begin{multline}
\label{Carestmoy}
\Bigl(\int_{\partial R(J,\delta)\times \partial R(J,\delta)}
\frac{|f(t)-f(t')|^2}{(\Lambda(t,t'))^2}d\Lambda(t)d\Lambda(t')\Bigr)^{1/2}
\\ \le C \bigl(\|\partial f\|_{L^2(R(J,\delta))}+\|\bar \partial f\|_{L^2(R(J,\delta))}\bigr).
\end{multline}
\end{lem}

\begin{proof}  Pick  $\delta\in(0,\delta_0]$, and write $e^{ia}$, $e^{ib}$
for the endpoints of $J$ with $a<b$ and $|a-b|<2\pi$.
The map $\varphi(\rho,\theta):=(\rho\cos\theta,\rho\sin\theta)$
is a diffeomorphism from $R:=(1-\delta,1)\times (a,b)$
onto $R(J,\delta)$ satisfying
$|||D\varphi|||\leq1$ and 
$|||(D\varphi)^{-1}|||\le c/(1-\delta_0)$, where
$D\varphi$ indicates the derivative and $|||\cdot|||$ is the operator norm. 
In particular,  
$\varphi^{-1}$ extends to a Lipschitz homeomorphism from
$\partial  R(J,\delta)$ onto $\partial R$
with Lipschitz constant depending only on $\delta_0$, and 
by the change of variable formula it is enough to show that if
$h:=f\circ\varphi$, then 
$$
\Bigl(\int_{\partial R\times \partial R}
\frac{|h(t)-h(t')|^2}{(\Lambda(t,t'))^2}d\Lambda(t)d\Lambda(t')\Bigr)^{1/2}
\le C \bigl(\|\partial h\|_{L^2(R)}+
\|\bar \partial h\|_{L^2(R)}\bigr),
$$
where the constant $C$ depends only on
$\Lambda(J)/\delta=2\pi(b-a)/\delta$. The result now follows from the fact that
if $p=2$ and $\Omega$ is a rectangle, then the constant in \eqref{majvarder}
depends only on the ratio of sidelengths, a fact which is obvious by 
homogeneity.
\end{proof}

\subsection{A multiplier theorem}

The next theorem is fundamental to our study of
$G^p_\alpha$ when $\alpha\in L^2(\D)$
but is also of independent interest. It is best stated in 
terms of multipliers.
We use the definition \eqref{defntmax} of the non-tangential maximal 
function ${\mathcal M}_\gamma f$. Denote by $\mathfrak{M}^{\gamma,p}$ 
the Banach space of complex-valued functions on $\D$ such that 
$\|\mathcal{M}_\gamma f\|_{L^p(\T)}<\infty$. 
Furthermore, we use the Banach space $\mathcal{H}^p$ of functions 
satisfying a Hardy condition, introduced in Section~\ref{Gpalpha}.

\begin{thm}
\label{somprodBN}
Let $\gamma\in(0,\pi/2)$ and $p\in[1,\infty)$. Given $f\in W^{1,2}_{0,\RR}(\D)$, the multiplication by $e^f$ is
continuous from $\mathfrak{M}^{\gamma,p}$ into $\mathcal{H}^p$.
More precisely, for any function  $g$ on $\D$, we have 
\begin{equation}
\label{estFM}
\sup_{0<\rho<1}\Bigl(\int_{\T_\rho}e^{p f(\xi)} \left\vert g(\xi)\right\vert^p |d\xi|\Bigr)^{1/p} <C \,\|{\mathcal M}_\gamma g\|_{L^p(\T)}, 
\end{equation}
where $C$ depends on $p$, $\gamma$, and on 
$\varepsilon>0$ so small that 
$\|\partial f\|_{L^2(Q_\varepsilon\cap\D)}<C'/p$ 
whenever $Q_\varepsilon$ is a square of sidelength 
$\varepsilon$, with $C'$ depending only
on  $\gamma$.
\end{thm}

\begin{proof} 
First, let $\rho\in(0,\sin\gamma)$. For $\zeta\in \T$, 
$\Gamma(\zeta,\gamma)$ contains $\T_\rho$ and we have
$$
\int_{\T_\rho}e^{pf(\xi)} \left\vert g(\xi)\right\vert^p |d\xi|
\le {\mathcal M}^p_\gamma g(\zeta)
\int_{\T_\rho} e^{p f(\xi)} |d\xi|.
$$
Averaging over $\zeta\in\T$ yields 
\begin{equation}
\label{separs}
\int_{\T_\rho}e^{p f(\xi)} \left\vert g(\xi)\right\vert^p |d\xi|
\le \frac{1}{2\pi}\Bigl(\int_\T{{\mathcal M}^p_\gamma g(\zeta)}\,|d\zeta|\Bigr)
\Bigl(\int_{\T_\rho}e^{p f(\xi)} |d\xi|\Bigr).
\end{equation}
By Lemma~\ref{bornewr} applied to $e^f$ in the place of $f$ with $\ell=p$ and
$ q:=2p/(p+1)$, we get
\begin{equation}
\label{separ}
\Bigl(\int_{\T_\rho}e^{p f(\xi)} |d\xi|\Bigr)^{1/p}\le c_0\|e^f\|_{W^{1,q}(\D)} 
\end{equation}
for some $c_0=c_0(p)$. Moreover, Lemma~\ref{LpexpW12}, Proposition~\ref{embexpS},  
H\"older's inequality, and the fact that $f$ is real-valued imply together that for some absolute constants $C_1,C_2$ we have 
\begin{multline}
\label{inegefW1q}
\|e^f\|_{W^{1,q}(\D)} =\|e^f\|_{L^q(\D)}+2\|\partial fe^f\|_{L^q(\D)}
\le\|e^{|f|}\|_{L^{2p}(\D)}\bigl(1+2\|\partial f\|_{L^2(\D)}\bigr)\\
\le C_1\bigl(1+\exp\bigl(C_2p\|\partial f\|^2_{L^2(\D)}\bigr)\bigr)
\bigl(1+\|\partial f\|_{L^2(\D)}\bigr).
\end{multline}
By \eqref{separs}, \eqref{separ}, and
\eqref{inegefW1q} we conclude that
\begin{equation}
\label{premsingamma}
\sup_{0<\rho  < \sin\gamma}
\Bigl(\int_{\T_\rho}e^{p f(\xi)} \left\vert g(\xi)\right\vert^p |d\xi|\Bigr)^{1/p}
\le C_0\|{\mathcal M}_\gamma g\|_{L^p(\T)} 
\end{equation}
for some $C_0=C_0(p,\|\partial f\|_{L^2(\D)})$. 

Assume next that $\rho\geq\sin\gamma$. Now 
$\Gamma(\zeta,\gamma)$ cuts out two disjoint
open arcs on $\T_\rho$  one of which
is centered at $\xi=\rho\zeta$. Denote this arc by $A_\xi$.
Its length $\Lambda(A_\xi)$ is independent of $\zeta$ and it is easy to 
check
that  $K_1(1-\rho)\leq\Lambda(A_\xi)\le K_2(1-\rho)$
for strictly positive numbers $K_1$, $K_2$ 
depending only on $\gamma$.
Take an integer $N_\rho$ in the interval $[4\pi\rho/\Lambda(A_\xi),4\pi\rho/\Lambda(A_\xi)+1)$, 
and divide $\T_\rho$ into $N_\rho$ semi open arcs of 
equal length, say $I_{\xi_1},\ldots,I_{\xi_{N_\rho}}$, centered at
equidistant points $\xi_1,\ldots,\xi_{N_\rho}\in \T_\rho$. 
Put $\zeta_j=\xi_j/\rho\in\T$, and 
consider the partition of $\T$ into $N_\rho$ semi open arcs
$J_{\zeta_j}:=I_{\xi_j}/\rho$ centered at $\zeta_j$. By construction, 
if $\zeta\in J_{\zeta_j}$, then
$I_{\xi_j}\subset\Gamma_{\zeta,\gamma}$. Consequently,
$$
\int_{I_{\xi_j}}e^{p f(\xi)} \left\vert g(\xi)\right\vert^p |d\xi|
\leq {\mathcal M}^p_\gamma g(\zeta)
\int_{I_{\xi_j}}e^{p f(\xi)} |d\xi|,
$$
and averaging over $\zeta\in J_{\zeta_j}$ gives us
$$
\int_{I_{\xi_j}} e^{p f(\xi)} 
\bigl\vert g(\xi)\bigr\vert^p |d\xi|\le \frac{1}{\Lambda(J_{\zeta_j})}\Bigl(\int_{J_{\zeta_j}}{{\mathcal M}^p_\gamma g(\zeta)}\,|d\zeta|
\Bigr)
\Bigl(\int_{I_{\xi_j}}e^{p f(\xi)} |d\xi|\Bigr).
$$
Since $\Lambda(J_{\zeta_j})=\Lambda(I_{\xi_j})/\rho$ we deduce upon summing over $j$ that
\begin{multline}
\label{shiftI}
\int_{\T_\rho} e^{p f(\xi)} \left\vert g(\xi)\right\vert^p |d\xi|
\\ \le \rho
\Bigl(\int_{\T}{{\mathcal M}^p_\gamma g(\zeta)}\,|d\zeta|\Bigr)
\sup_{1\le j\le N_\rho}\Bigl(\frac{1}{\Lambda(I_{\xi_j})}
\int_{I_{\xi_j}} e^{p f(\xi)} |d\xi|\Bigr).
\end{multline}
Let $R(J,\delta)$ be defined as in \eqref{defR}, and let $C$ be the constant in Lemma~\ref{ecmoyW120} associated to 
$\delta_0=1-\sin\gamma$ and  $\Lambda(J)/\delta=K_2/(2\sin\gamma)$;
note that $C$ depends only on $\gamma$. 
Since $\Lambda(J_{\zeta_j})/(1-\rho)\le K_2/(2\sin\gamma)$, the arc
$J'_{\zeta_j}\subset \T$ of  length 
$(1-\rho)K_2/(2\sin\gamma)$  centered at $\zeta_j$ does contain $J_{\zeta_j}$.
Therefore, $R(J_{\zeta_j},1-\rho)$ is contained in $R(J'_{\zeta_j},1-\rho)$ and 
$I_{\xi_j}$ is contained in $I'_{\xi_j}:=J'_{\zeta_j}/\rho$. Hence, 
\eqref{Carestmoy} {\it a fortiori} implies for some $K$ depending only on $\gamma$ that 
\begin{multline}
\label{contprime}
\Bigl(
\int_{I_{\xi_j}\times I_{\xi_j} }
\frac{|f(t)-
f(t')|^2}{(\Lambda(t,t'))^2}d\Lambda(t)d\Lambda(t')\Bigr)^{1/2}
\\ \le K 
\bigl(\|\partial f\|_{L^2(R(J'_{\zeta_j},1-\rho))}+
\|\bar \partial f\|_{L^2(R(J'_{\zeta_j},1-\rho))}\bigr).
\end{multline}
Now, it is elementary to check that $R(J'_{\zeta_j}, 1-\rho)$ is contained in a square of sidelength $K_3(1-\rho)$ 
(where $K_3$ depends only on $\gamma$), one side of which is tangent to $\T$ at $\zeta_j$. 
So, if we let $\varepsilon_1$ be so small  that
$\|\partial \tilde{f}\|_{L^2(Q_{\varepsilon_1})}<1/(8Kep)$ 
whenever $Q_{\varepsilon_1}$ is a square of sidelength 
$\varepsilon_1$, we get (since $f$ is real-valued) 
that 
\begin{multline}
\label{estderR}
\|\partial f\|_{L^2(R(J'_{\zeta_j},1-\rho))}+
\|\bar \partial f\|_{L^2(R(J'_{\zeta_j},1-\rho))}\leq \frac{1}{4Kep},
\\ \max(\sin\gamma, 1-\varepsilon_1/K_3)=\rho_0\le\rho<1.
\end{multline}
Then, from \eqref{VMOW1/2},  \eqref{contprime}, and \eqref{estderR},   
we see that for all subarcs $I\subset I_{\xi_j}$ 
\begin{equation}
\label{chaineosc}
\begin{array}{lll}
\frac{1}{\Lambda(I)}\int_{I}|f-f_{I}|\,d\Lambda&\leq&
\left(
\int_{I\times I}\frac{|f(t)-f(t')|^2}{(\Lambda(t,t'))^2}d\Lambda(t)d\Lambda(t')\right)^{1/2}\\
&\leq&
\left(
\int_{{I_{\xi_j}}\times I_{\xi_j}}\frac{|f(t)-f(t')|^2}{(\Lambda(t,t'))^2}
d\Lambda(t)d\Lambda(t')\right)^{1/2}\\
&\leq&1/4ep,\qquad \rho_0\le\rho<1.
\end{array}
\end{equation}
If we let $h:=\tr_{\T_\rho}f$, inequality \eqref{chaineosc} asserts that
\begin{equation}
\label{majoscIxi}
M_{h}(I_{\xi_j})\leq 1/4ep, \qquad 
\rho_0\leq\rho<1,
\end{equation}
with $M_{h}(I_{\xi_j})$ defined by \eqref{defMhG} 
where we set $\Gamma$ to be $\T_\rho$. By Lemma~\ref{expsumVMOG}, \eqref{majoscIxi}, and H\"older's inequality, for 
$\rho\in[\rho_0,1)$ we have   
\begin{equation}
\label{estGR}
\Bigl(\frac{1}{\Lambda(I_{\xi_j})}
\int_{I_{\xi_j}} e^{pf(\xi)} |d\xi|\Bigr)^{1/p}
\le (1+e)^{1/p}\exp\Bigl(
\Bigl\vert\frac{1}{\Lambda(I_{\xi_j})}\int_{I_{\xi_j}}
f(\zeta) \,|d\zeta|\Bigr\vert\Bigr).
\end{equation}
In another connection, keeping in mind \eqref{estderR} and the inclusion
$R(J_{\zeta_j},1-\rho)\subset R(J'_{\zeta_j},1-\rho)$,
an application of \eqref{estmoy} yields
\begin{equation}
\label{expm}
\Bigl\vert\frac{1}{\Lambda(I_{\xi_j})}\int_{I_{\xi_j}}
f(\zeta) \,|d\zeta|\Bigr\vert\le 
\frac{(1-\rho)^{1/2}}{(\Lambda(I_{\xi_j}))^{1/2}}\,\frac{1}{4Kep}.
\end{equation}
Put $\rho_1:=\max(\rho_0,K_1/(K_1+\pi))$, and assume for a while that
$\rho\geq\rho_1$; in particular, $\pi\rho/(K_1(1-\rho))>1$,
and therefore
\begin{equation}
\label{lbI}
\Lambda(I_{\xi_j})\ge\frac{2\pi\rho}{1+4\pi\rho/|A_\xi|}\ge
\frac{2\pi\rho}{1+4\pi\rho/(K_1(1-\rho))}
\ge \frac{K_1(1-\rho)}{3}.
\end{equation}
Using together \eqref{expm} and \eqref{lbI}, we obtain
\begin{equation}
\label{expm1}
\Bigl\vert\frac{1}{\Lambda(I_{\xi_j})}\int_{I_{\xi_j}}
f(\zeta) \,|d\zeta|\Bigr\vert\le 
\frac{\sqrt{3}}{4\sqrt{K_1}Kep},\qquad \rho_1\leq\rho<1.
\end{equation}
Plugging \eqref{expm1} in the right hand side of \eqref{estGR} and
using \eqref{shiftI} now gives us 
$$
\sup_{\rho_1\le\rho<1}
\Bigl(\int_{\T_\rho}e^{p f(\xi)} \left\vert g(\xi)\right\vert^p |d\xi|\Bigr)^{1/p}
\le (1+e)^{1/p}\exp\Bigl(\frac{\sqrt{3}}{4\sqrt{K_1}Kep}\Bigr)
\|{\mathcal M}_\gamma g\|_{L^p(\T)}. 
$$
To obtain \eqref{estFM}, it remains to treat 
the case 
$\rho\in[\sin\gamma,\rho_1)$ when the 
latter interval is nonempty. First, in this range of $\rho$,
the first two inequalities in \eqref{lbI} imply that
\begin{equation}
\label{longmaj}
\Lambda(I_{\xi_j})\ge c(\gamma,\rho_1).
\end{equation}
On the other hand,
\eqref{separ} and  \eqref{inegefW1q} give us that
\begin{equation}
\label{separf}
\Bigl(\int_{I_{\xi_j}}e^{p f(\xi)} |d\xi|\Bigr)^{1/p}\le
\Bigl(\int_{\T_\rho}e^{p f(\xi)} |d\xi|\Bigr)^{1/p}
\le  
C_0
\end{equation}
with $C_0$ as in \eqref{premsingamma}.
Therefore by \eqref{longmaj} and \eqref{separf} we have that
$$
\Bigl(\frac{1}{\Lambda(I_{\xi_j})}\int_{I_{\xi_j}}e^{p f(\xi)} |d\xi|\Bigr)^{1/p}\le 
[c(\gamma,\rho_1)]^{-1/p}
C_0,\qquad \sin\gamma\leq \rho<\rho_1,
$$
and using this in \eqref{shiftI} completes the proof.
\end{proof}

\begin{cor}
\label{condHm}
If $w=e^sF$, where $s\in W^{1,2}(\D)$ with $\mbox{\rm Re}\,\tr_\T s\equiv0$ and
$F\in H^p$, then 
\[ 
\sup_{0<\rho  < 1}
 \left(\frac{1}{2\pi}\int_{\T_\rho} \left\vert
 w(\xi)\right\vert^p |d\xi|\right)^{1/p} <+\infty.
\]
\end{cor}

\begin{proof} This follows from \eqref{ntborne} and
Theorem~\ref{somprodBN} applied with $f={\rm Re}\,s$
and $g=e^{i{\rm Im}\,s}F$.
\end{proof}

\end{document}